\documentclass[11pt]{amsart}
\usepackage{amsmath,amsthm,amsfonts,amscd,amssymb,eucal,latexsym,mathrsfs}
\usepackage[all]{xy}

\newtheorem{theorem}{Theorem}[section]
\newtheorem{corollary}[theorem]{Corollary}
\newtheorem{lemma}[theorem]{Lemma}
\newtheorem{proposition}[theorem]{Proposition}

\theoremstyle{definition}
\newtheorem{definition}[theorem]{Definition}
\newtheorem{remark}[theorem]{Remark}

\newcommand{\Tr}{{\rm Tr}}

\newcommand{\spn}{{\rm span}}

\newcommand{\Aut}{{\rm Aut}}

\newcommand{\Ad}{{\rm Ad}\,}

\newcommand{\id}{{\rm id}}

\newcommand{\red}{{\lambda}}

\newcommand{\cB}{{\mathcal B}}
\newcommand{\cH}{{\mathcal H}}
\newcommand{\cK}{{\mathcal K}}

\newcommand{\cU}{{\mathcal U}}
\newcommand{\cO}{{\mathcal O}}

\newcommand{\cP}{{\mathcal P}}

\newcommand{\Cb}{{\mathbb C}}
\newcommand{\Zb}{{\mathbb Z}}

\newcommand{\Rb}{{\mathbb R}}
\newcommand{\Nb}{{\mathbb N}}

\newcommand{\tr}{{\rm tr}}

\newcommand{\op}{{\rm op}}

\newcommand{\mx}{{\rm max}}

\newcommand{\Rep}{{\rm Rep}}
\newcommand{\Irr}{{\rm Irr}}

\newcommand{\orb}{{\mathcal O}}

\newcommand{\Act}{{\rm Act}}

\newcommand{\unit}{{\mathbf 1}}

\newcommand{\iso}{{\rm i}}
\newcommand{\WM}{{\rm WM}}
\newcommand{\FWM}{{\rm FWM}}
\newcommand{\FE}{{\rm FE}}

\newcommand{\Erg}{{\rm Erg}}

\newcommand{\SL}{{\rm SL}}
\newcommand{\Fr}{{\rm Fr}}

\allowdisplaybreaks

\begin{document}

\title[Turbulence]{Turbulence, representations, and trace-preserving actions}

\author{David Kerr}
\author{Hanfeng Li}
\author{Mika\"el Pichot}
\address{\hskip-\parindent
David Kerr, Department of Mathematics, Texas A{\&}M University,
College Station TX 77843-3368, U.S.A.}
\email{kerr@math.tamu.edu}

\address{\hskip-\parindent
Hanfeng Li, Department of Mathematics, SUNY at Buffalo,
Buffalo NY 14260-2900, U.S.A.}
\email{hfli@math.buffalo.edu}

\address{\hskip-\parindent
Mika\"el Pichot, Department of Mathematical Sciences,
University of Tokyo, 3-8-1 Komaba, Tokyo 153-8914, Japan}
\email{pichot@ms.u-tokyo.ac.jp}

\date{August 11, 2008}

\begin{abstract}
We establish criteria for turbulence in certain spaces of $C^*$-algebra representations
and apply this to the problem of
nonclassifiability by countable structures for group actions on a standard atomless
probability space $(X,\mu )$ and on the hyperfinite II$_1$ factor $R$.
We also prove that the conjugacy action
on the space of free actions of a countably infinite amenable group on $R$ is turbulent,
and that the conjugacy action on the space of ergodic
measure-preserving flows on $(X,\mu )$ is generically turbulent.
\end{abstract}

\maketitle

\section{Introduction}

Descriptive set theory provides a natural framework for the study of the complexity of classification
problems in analysis and dynamics \cite{Kechris,COER}. Often one has a collection of objects that can be viewed
as elements in a Polish space $X$ and an equivalence relation $E$ on $X$ encoding the isomorphism
relation between the objects. Consider for example the set of unitary operators on a separable Hilbert space with the
strong operator topology or the set of measure-preserving transformations of a standard probability
space with the weak topology, each under the relation of conjugacy. We may then attempt to gauge
the complexity of $E$ by the way it relates descriptively to other equivalence relations. Given another
equivalence relation $F$ on a standard Borel space $Y$, one says that $E$ is {\it Borel reducible} to $F$
if there is a Borel map $\theta : X\to Y$ such that, for all $x_1 , x_2 \in X$, $x_1 Ex_2$ if and only if
$\theta (x_1 )F\theta (x_2 )$. If we can Borel reduce $E$ to a relation on objects which are in some sense
better understood, we may reasonably claim to have a classification theory.

The relation $E$ is said to be {\it smooth} if it can be Borel reduced to equality on $\Rb$, i.e.,
if we can assign numerical invariants in a Borel manner. By a theorem of Glimm,
the space of irreducible representations of a separable $C^*$-algebra $A$ is smooth
precisely when $A$ is type I (see Section~6.8 of \cite{Ped}). The theorem of Ornstein asserting that entropy
is a complete invariant for Bernoulli shifts provides another example of smoothness \cite{Orn}.
To show that a Borel equivalence relation $E$ on $X$ is not smooth, it suffices to demonstrate the existence
of a Borel probability measure on $X$ which is ergodic (i.e., every invariant Borel set has measure $0$ or $1$)
and is zero on every equivalence class. The relation $E_0$ of tail equivalence
on $\{ 0,1 \}^\Nb$ satisfies this proper ergodicity condition in a prototypical way, and
indeed when $E$ is Borel the continuous embeddability of $E_0$ into $E$ is a universal obstruction to
smoothness \cite{HKL}. There is also a topological version of the proper ergodicity
obstruction to smoothness via Baire category in the case that $E$ arises as the orbit equivalence
relation of the continuous action of a Polish group on $X$, namely that every equivalence class
is both dense and meager (see Section~3.1 of \cite{COER}).

At a higher level of complexity is the notion of classification by countable structures,
which means that $E$ can be Borel reduced to the isomorphism relation on the space of
countable structures of some countable language as implemented by the logic action of the
infinite permutation group $S_\infty$ with its unique Polish topology \cite[Defn.\ 2.38]{COER}.
Equivalently, $E$ can be Borel reduced to the orbit equivalence relation
of a Borel action of $S_\infty$ on a Polish space \cite[Sect.\ 2.7]{BecKec}.
Non-smooth examples of this are the Halmos-von Neumann
classification of discrete spectrum transformations by their sets of eigenvalues \cite{HvN} and
the Giordano-Putnam-Skau classification of minimal homeomorphisms of the Cantor set up
to strong orbit equivalence by countable ordered Abelian groups \cite{GPS}.
Note that the isomorphism relation on any type of countable algebraic structure
can be encoded as a continuous $S_\infty$-action
on a Polish space (see Example~2 in \cite{FW}).

In analogy with the topological proper ergodicity obstruction to smoothness, Hjorth developed the
notion of turbulence as a means for demonstrating nonclassifiability by countable structures \cite{COER}.
Let $G$ be a Polish group acting continuously on a Polish space $X$. For an $x\in X$ and
open sets $U\subseteq X$ and $V\subseteq G$ containing $x$ and $e$, respectively, we define
the local orbit $\cO (x,U,V)$ as the set of all $y\in U$ for which
there are $g_1 , g_2 , \dots , g_n \in V$ such that
$g_k g_{k-1} \cdots g_1 x \in U$ for each $k=1,\dots , n-1$ and $g_n g_{n-1} \cdots g_1 x = y$.
A point $x\in X$ is {\it turbulent} if for all nonempty open sets $U\subseteq X$ and $V\subseteq G$
containing $x$ and $e$, respectively, the closure of $\cO (x,U,V)$ has nonempty interior.
The action is {\it turbulent} if every orbit is dense and meager and every point is turbulent.
Section~3.2 of \cite{COER} shows that if the action of $G$ on $X$
is turbulent then whenever $F$ is an equivalence relation
arising from a continuous action of $S_\infty$ on a Polish space $Y$
and $\theta : X\to Y$ is a Baire measurable function such that $x_1 Ex_2$ implies $\theta (x_1 )F\theta (x_2 )$,
there exists a comeager set $C\subseteq X$ such that $\theta (x_1 )F\theta (x_2 )$ for all $x_1 , x_2 \in C$.
As a consequence the orbit equivalence relation on $X$ does not admit classification by countable structures.
In fact to obtain this conclusion it suffices to show that the action is {\it generically turbulent},
which can be expressed by saying that some orbit is dense, every orbit is meager,
and some point is turbulent (see Definiton~3.20 and Theorem~3.21 in \cite{COER}).
By Theorem~3.21 of \cite{COER}, if the action of $G$ on $X$ is
generically turbulent then there is a $G$-invariant dense $G_\delta$ subset of $X$ on which the action
is turbulent.

Turbulence has now been established in several situations.
Hjorth showed in \cite{NSIDGR} that if $G$ is a countably infinite
group which is not a finite extension of an Abelian group (which in this case is equivalent to $G$ not being
type I by a result of Thoma) then the space of irreducible representations of $G$ on a separable infinite-dimensional
Hilbert space $\cH$ under the conjugation action of the unitary group $\cU (\cH )$ admits an invariant $G_\delta$
subset on which the action is turbulent. Hjorth's argument yields the same conclusion more generally for
the space of irreducible representations of any separable non-type I
$C^*$-algebra on $\cH$. Within the type I realm, Kechris and Sofronidis established generic turbulence
for the conjugation actions of $\cU (\cH )$ on itself and on the space of self-adjoint operators
of norm at most one with the strong topology \cite{KS}.

Suppose now that $G$ is a countably infinite group and consider the Polish space\linebreak 
$\Act (G,X,\mu )$ of actions
of $G$ by measure-preserving transformations on a standard atomless probability space $(X,\mu )$
under the weak topology, with the conjugation action of $\Aut (X,\mu ) = \Act (\Zb ,X,\mu )$. Hjorth constructed
a turbulent action which Borel reduces to the conjugacy relation on the space of ergodic automorphisms
in $\Aut (G,X,\mu )$ \cite{OIMPT} and used turbulence
in spaces of irreducible representations to show nonclassifiability by countable structures for the subspace of
free weakly mixing actions
when $G$ is not a finite extension of an Abelian group (see Theorem~13.7 in \cite{Kechris}).
Foreman and Weiss proved
that the action of $\Aut (X,\mu )$ on the space of free ergodic actions in
$\Act (G,X,\mu )$ is turbulent when $G$ is amenable, using entropy
and disjointness to obtain the meagerness of orbits and the Rokhlin lemma and orbit equivalence to
verify that every point is turbulent \cite{FW}. Free weakly mixing actions of any countably infinite $G$
considered up to unitary conjugacy also do not admit classification by countable structures
\cite[Thm.\ 13.8]{Kechris}.

One of the main goals of the present paper is to develop a general spectral approach to the identification
of turbulent behaviour in spaces of representations and actions.
We prove that, for a separable $C^*$-algebra
$A$ and a separable infinite-dimensional Hilbert space $\cH$, the action of $\cU (\cH )$ on the Polish space
of faithful essential nondegenerate representations of $A$ on $\cH$ has the property that every point is turbulent
and has dense orbit, while the meagerness of all orbits is equivalent to the isolated points
of the spectrum $\hat{A}$ not being dense, so that the action is turbulent precisely in this case.
The action of $\cU (\cH )$ on the Polish space
of all nondegenerate representations of $A$ on $\cH$ is turbulent precisely when $A$ is simple and
not isomorphic to the compact operators on some Hilbert space. Furthermore, the orbit equivalence relation
on the space of nondegenerate representations does not admit classification by countable structures
as soon as $\hat{A}$ is uncountable (if $\hat{A}$ is
countable then the classification of nondegenerate representations on $\cH$ is a matter of counting
multiplicities of irreducible subrepresentations and hence is smooth).

This spectral picture leads in particular to a unified proof of nonclassifiability by countable
structures for free weakly mixing actions of countably infinite $G$ that does not rely on the
type I/non-type I dichotomy, and also allows us to extend the conclusion to
weakly mixing actions of many nondiscrete $G$ of type I, such as $\Rb^d$ and $\SL (2,\Rb )$.
We show moreover that the same nonclassifiability statements hold if we replace $(X,\mu )$
with the hyperfinite II$_1$ factor $R$.
What is of particular interest about the noncommutative context is the fact that,
under amenability assumptions, actions on a factor can be classified up
to cocycle conjugacy by cohomological invariants. For actions of countable amenable groups on $R$
this was done by Ocneanu \cite[Thm.\ 2.6]{Ocneanu}, extending the fundamental work of Connes
on single automorphisms \cite{OCC}. For finite groups one can go further and
produce a classification up to conjugacy, as was done by Connes in the periodic case \cite{PA} and Jones
in general \cite{Jones}.

In the case that the acting group $G$ is countably infinite and amenable, we prove that
the action of the automorphism group of $R$ on the space of free $G$-actions on $R$ is
turbulent. To obtain the meagerness of orbits we follow the idea of Foreman and Weiss
of using entropy and disjointness, although in the noncommutative situation a different
technical perspective is required. We show for general second countable locally compact $G$
that there exists a turbulent point with dense orbit, and deduce from this that every point
is turbulent when $G$ is countably infinite and amenable by applying Ocneanu's result
that any two free actions are cocycle conjugate in this case, with bounds on the cocycle
\cite[Thm.\ 1.4]{Ocneanu}.
Our method for demonstrating the existence of a turbulent point with dense orbit also
works in the commutative situation, yielding a proof that works equally well for nondiscrete $G$
and does not involve orbit equivalence (compare \cite{FW} and Section~5 in \cite{Kechris}).
Using this we deduce that the action of $\Aut (X,\mu )$ on the space of ergodic
measure-preserving flows on $(X,\mu )$ is generically turbulent.

The paper contains five sections beyond the introduction.
Section~2 contains results on turbulence in spaces of $C^*$-algebra representations, while Section~3 discusses
the ramifications of these for group representations. In Section~4 we discuss freeness and weak mixing
and establish our nonclassifiability results
for actions based on the spectral analysis of Section~2. Section~5 contains the proof of turbulence in
the space of free actions of a countably infinite amenable group on $R$. Finally, in
Section~6 we show generic turbulence in the space of ergodic measure-preserving flows.
\medskip

\noindent{\it Acknowledgements.} D.K. was partially supported by NSF grant DMS-0600907.
H.L. was partially supported by NSF grant DMS-0701414.
M.P. was supported by the EPDI and a JSPS fellowship
for European researchers. M.P. is grateful to the Max-Planck
Institut f\"ur Mathematik for hospitality and to Yasuyuki Kawahigashi
for hosting his stay at the University of Tokyo. The initial stages of this
work were carried out during a visit of M.P. to Texas A\&M University in July 2007.


\section{Representations of $C^*$-algebras}\label{S-rep}

Let $A$ be a separable $C^*$-algebra and $\cH$ a separable infinite-dimensional Hilbert space.
A representation $\pi : A \to \cB (\cH )$ is said to be {\it essential} if
$\pi (A) \cap \cK (\cH ) = \{ 0 \}$, where $\cK (\cH )$ denotes the $C^*$-algebra of
compact operators on $\cH$.
We say that $\pi$ is {\it nondegenerate} if for every nonzero
vector $\xi\in\cH$ there is an $a\in A$ such that $\pi (a)\xi \neq 0$. If
$\{ h_\eta \}_\eta$ is an approximate unit for $A$, then $\pi$ is nondegenerate if and only if
$\pi (h_\eta )$ tends to the identity operator on $\cH$ in the strong operator topology.
In particular, if $A$ is unital then $\pi$ is nondegenerate if and only if it is unital.
By Voiculescu's theorem, any two faithful essential nondegenerate representations $\pi_1$ and $\pi_2$
of $A$ on separable infinite-dimensional Hilbert spaces $\cH_1$ and $\cH_2$, respectively, are
approximately unitarily equivalent in the sense that there exists a sequence of unitary operators
$U_n : \cH_1 \to \cH_2$ such that
\[ \lim_{n\to\infty} \big\| U_n \pi_1 (a) U_n^{-1} - \pi_2 (a) \big\| = 0 \]
for all $a\in A$, and every representation of
$A$ is approximately unitarily equivalent to a direct sum of
irreducible representations \cite{NCWvNT}. (see also \cite{Dav,QD}).

We write $\Rep (A,\cH )$ for the Polish space of all nondegenerate representations of $A$ on $\cH$
whose topology has as a basis the sets
\[ Y_{\pi ,F, \Omega ,\varepsilon} =
\{ \rho\in\Rep (A,\cH ) : | \langle (\rho (a) - \pi (a))\xi , \zeta\rangle | < \varepsilon
\text{ for all } a\in F \text{ and } \xi , \zeta \in\Omega \} \]
where $\pi\in\Rep (G,\cH )$, $F$ is a finite subset of $A$, $\Omega$ is a
finite subset of $\cH$, and $\varepsilon > 0$.
Sets of the form
\[ \{ \rho\in\Rep (A,\cH ) : \| (\rho (a) - \pi (a))\xi \| < \varepsilon
\text{ for all } a\in F \text{ and } \xi \in\Omega \} \]
with the same type of $\pi$, $F$, $\Omega$, and $\varepsilon$ also form a basis for the topology.
We write $\cU (\cH )$ for the group of unitary operators on $\cH$ equipped with the relative
strong operator (equivalently, relative weak operator) topology, under which it is a Polish
group. We will be concerned with the continuous action $(U,\pi ) \mapsto \Ad U \circ\pi$ of
$\cU (\cH )$ on $\Rep (A,\cH )$.

The spectrum $\hat{A}$ of $A$ is defined as the set of unitary equivalence
classes of irreducible representations of $A$ equipped with the topology
under which the canonical map from $\hat{A}$ onto the primitive ideal space of $A$
with the Jacobson topology is open and continuous \cite[Sect.\ 4.1]{Ped}.
As usual we identify elements in $\hat{A}$ with their representatives.
As we are assuming $A$ to be separable, the topology on $\hat{A}$ is second countable
\cite[Prop.\ 3.3.4]{Dix}. This topology can also be described in terms of
weak containment or, in the case that $A$ has no finite-dimensional
irreducible representations, as the quotient topology on unitary equivalence classes of
irreducible representations in $\Rep (A,\cH )$ (if $A$ has finite-dimensional
irreducible representations then one can stabilize and use the canonical homeomorphism
from $\widehat{A\otimes\cK}$ to $\hat{A}$) \cite[Sects.\ 3.4 and 3.5]{Dix}\cite[II.6.5.16]{OA}.

In the proof of the following lemma we use a rotation trick as in \cite{NSIDGR}.
For a closed linear subspace $E$ of $\cH$ we write $P_E$ for the orthogonal projection of $\cH$ onto $E$.

\begin{lemma}\label{L-turbpoint}
Let $\pi\in\Rep (A,\cH )$. If $\pi$ is faithful and essential then it has dense orbit and is a
turbulent point for the action of $\cU (\cH )$. If $\pi$ is not faithful or not essential
then its orbit is nowhere dense.
\end{lemma}

\begin{proof}
Suppose first that $\pi$ is faithful and essential.
To establish turbulence, let $Y$ be a neighbourhood of $\pi$ in $\Rep (A,\cH )$
and $Z$ a neighbourhood of $\mathbf{1}$ in $\cU (\cH )$. We will show that the closure of
the local orbit $\orb (\pi , Y, Z)$ has nonempty interior. We may suppose by shrinking
$Y$ and $Z$ if necessary that $Y = Y_{\pi , K, \Omega , \varepsilon}$ and
$Z = Z_{\mathbf{1} ,\Omega , \varepsilon}$ where $K$ is a finite subset of $A$, $\Omega$ is a finite subset
of the unit ball of $\cH$,
and $\varepsilon > 0$. Suppose that we are given a $\sigma\in Y$.
Let $L$ be a finite subset of $A$, $\Upsilon$ a finite subset of the unit ball of $\cH$, and $\delta > 0$. We will
construct a norm continuous path of unitaries $W_t \in \cU (\cH )$ for $t\in [0,\pi /2]$ such that $W_0 = \mathbf{1}$,
$\Ad W_t \circ\pi\in Y$ for every $t\in [0,\pi /2]$, and
$\Ad W_{\pi /2} \circ\pi\in Y_{\sigma , L, \Upsilon , \delta}$. This will show that
$\sigma$ lies in the closure of $\orb (\pi , Y, Z)$, since the continuity of the path permits us
to find $t_0 = 0 < t_1 < t_2 < \cdots < t_m = \pi /2$ such that $W_{t_i} W_{t_{i-1}}^{-1} \in Z$
for each $i=1, \dots ,m$.

We may assume that $\Upsilon$ contains $\Omega$ and that $\delta$ is small enough so that
$\| \sigma (a)\xi - \pi (a) \xi \| + \delta < \varepsilon$ for all $a\in K$ and $\xi\in\Omega$.
Let $E$ be the subspace spanned by $\Upsilon\cup\pi (K)\Upsilon\cup\sigma (L)\Upsilon$.
Since $\pi$ is faithful and essential, by the matrix version of 
Glimm's lemma \cite{BunSal76} (see also Lemma~II.5.2 in \cite{Dav} and the
paragraph following it) there is an isometry $V : E\to \cH$ such that $VE \perp\spn (E\cup \pi (K\cup L)^* E)$
and $\| V^* \pi (a) V - P_E \sigma (a) |_E \| < \delta$ for all $a\in K\cup L$.
For each $t\in [0,\pi /2]$ let $W_t$ be the unitary operator on $\cH$ which is the
identity on $(E\oplus VE )^\perp$ and acts on $E\oplus VE$ in $2\times 2$ block form as
\[ \left[ \begin{matrix} \cos (t)\mathbf{1} & \sin (t)V^* \\ -\sin (t)V & \cos (t)\mathbf{1} \end{matrix} \right] . \]

Let $t\in [0,\pi /2]$, $a\in K\cup L$, and $\xi\in\Upsilon$. Since
$\| P_{VE} \pi (a) V\xi - V\sigma (a)\xi \| < \delta$ and $\pi (a) V\xi \in E^\perp$
we have, writing $c = \cos t$ and $s = \sin t$,
\begin{align*}
W_t \pi (a) W_t^{-1} \xi &= W_t \big( c\pi (a)\xi + s\pi (a) V\xi \big) \\
&\approx_\delta W_t \big( c\pi (a)\xi + sV\sigma (a) \xi + sP_{E\oplus VE}^\perp \pi (a) V\xi \big) \\
&= c^2 \pi (a)\xi + s^2 \sigma (a) \xi + csV(-\pi (a)\xi + \sigma (a)\xi )
+ sP_{E\oplus VE}^\perp \pi (a) V\xi .
\end{align*}
It follows that if $a\in K$ and $\xi\in\Omega$ then
\begin{align*}
\| P_E (W_t \pi (a) W_t^{-1} \xi - \pi (a)\xi ) \| &\leq
\| c^2 \pi (a)\xi + s^2 \sigma (a)\xi - \pi (a) \xi \| + \delta \\
&= s^2 \| \sigma (a)\xi - \pi (a) \xi \| + \delta \\
&< \varepsilon .
\end{align*}
and so $\Ad W_t \circ\pi\in Y$ since $\Omega$ is contained in the unit ball of $E$.
In the case $t = \pi /2$ and $a\in L$
we obtain $\| P_E (W_{\pi /2} \pi (a) W_{\pi /2}^{-1} \xi - \sigma (a)\xi ) \| < \delta$
and so $\Ad W_{\pi /2} \circ\pi\in Y_{\sigma , L, \Upsilon , \delta}$ since $\Upsilon$ is
contained in the unit ball of $E$.
We thus conclude that $\sigma$ lies in the closure of $\orb (\pi , Y, Z)$ and hence that
$\pi$ is a turbulent point for the action of $\cU (\cH )$. Moreover, since we have shown that
$Y\subseteq \overline{\orb (\pi , Y, Z)}$ for any $Y$ of the form $Y_{\pi , K, \Omega , \varepsilon}$,
we see that the orbit of $\sigma$ is dense.

Now if $\pi$ is any faithful representation in $\Rep (A,\cH )$ then for every $a\in A$
we evidently have $\sigma (a) \neq 0$ for all $\sigma$ in some neighbourhood of
$\pi$. Since the faithful representations are dense in $\Rep (A,\cH )$ by the first paragraph,
we deduce that the orbit of every nonfaithful representation is nowhere dense.

Suppose finally that $\pi$ is a representation in $\Rep (A,\cH )$ which is not essential.
Then $\pi (A)$ contains a nonzero positive compact operator, and hence by the functional
calculus there is an $a\in A$ such that $\pi (a)$ is nonzero and of finite rank.
Take a faithful essential representation $\sigma\in\Rep (G,\cH )$ (for example, a
representation unitarily equivalent to $\rho^{\oplus\Nb}$ for any faithful $\rho\in\Rep (G,\cH )$).
It is readily seen that for every $\sigma'$ in some neighbourhood of
$\sigma$ the range of $\sigma' (a)$
has dimension larger than the rank of $\pi (a)$. Since the orbit of $\sigma$ is dense
as shown above, we conclude that the orbit of $\pi$ is nowhere dense.
\end{proof}

Since for an action of a group on a second countable topological space
the set of points with dense orbit is a $G_\delta$, we obtain from Lemma~\ref{L-turbpoint}
the following.

\begin{lemma}\label{L-fe}
The set of faithful essential representations in $\Rep (A,\cH )$ is a dense $G_\delta$.
\end{lemma}

The following was observed for unitary representations of countable groups
in Proposition~H.2 of \cite{Kechris} but the same argument applies more generally.

\begin{lemma}\label{L-wc}
Let $\pi , \sigma\in\Rep (A,\cH )$ and let $\rho$ be an element of $\Rep (A,\cH )$
unitarily conjugate to $\sigma^{\oplus\Nb}$. Then $\pi$ is weakly contained in $\sigma$ if and only
if $\pi$ lies in the orbit closure of $\rho$.
\end{lemma}

For each open set $U\subseteq\hat{A}$ there is a closed ideal $I\subseteq A$ such that $U$
is equal to the set $\hat{A}^I$ of all $\sigma\in\hat{A}$ for which $\sigma (I) \neq 0$, and the
restriction map $\hat{A}^I \to \hat{I}$ is a homeomorphism \cite[Sect.\ 3.2]{Dix}. This sets up
a bijective correspondence between the open subsets of $\hat{A}$ and the closed ideals
of $A$, and for each closed ideal $I$ we regard $\hat{I}$ as an open subset of $\hat{A}$.

We write $\hat{A}_\iso$ for the set of isolated points in $\hat{A}$. This set is countable
because $\hat{A}$ is second countable, which follows from the separability of $A$
\cite[Prop.\ 3.3.4]{Dix}.

\begin{lemma}\label{L-not dense}
Let $A$ be a separable $C^*$-algebra such that $\hat{A}_\iso$ is not dense in $\hat{A}$.
Let $I$ be the closed ideal of $A$ for which $\hat{I}$ is the complement of the closure of $\hat{A}_\iso$.
Let $\pi\in  {\rm Rep}(A, \cH)$ and set $\cH_I=\overline{\pi(I)\cH}$. Then the set
$(\pi|_{\cH_I})^{\perp}$ of representations in $\Rep (A,\cH )$ which are disjoint from $\pi |_{\cH_I}$
is a dense $G_{\delta}$.
\end{lemma}

\begin{proof}
By Lemma~3.7.3 of \cite{Dix}, $(\pi|_{\cH_I})^{\perp}$ is a $G_{\delta}$ subset of ${\rm Rep}(A, \cH)$.
By Voiculescu's theorem \cite[Cor.\ 1.6]{NCWvNT}, the orbit closure of a given element of ${\rm Rep}(A, \cH)$
is the same as the orbit closure of some $\pi'=\pi_0\oplus \bigoplus_{k\in K}\pi_k$ with $\pi_0(I)=\{0\}$,
$K$ a countable index set, and $\pi_k \in\hat{I}$ for each $k\in K$. To show that
$(\pi|_{\cH_I})^{\perp}$ is dense, it suffices to show that $\pi'$ is
in the orbit closure of some element in $(\pi|_{\cH_I})^{\perp}$.

We first argue that each nonempty open subset of $\hat{I}$ is uncountable.
Indeed suppose to the contrary that there is a countable open set $U\subseteq \hat{I}$.
Let $I_0$ be the closed ideal of $A$ such that $\hat{I_0} = U$.
By Lemma~1.3 of \cite{Wang}, $\hat{J}$ has an isolated point. This point is also isolated
when viewed as an element of $\hat{A}$, yielding a contradiction.

Since $\cH_I$ is separable, the set $D$ of irreducible subrepresentations of $\pi|_{\cH_I}$
is countable. Thus, since $\hat{I}$ is second countable and
each of its nonempty open subsets is uncountable, we can construct a countable
set $E\subseteq\hat{I} \setminus D$ which contains $\pi_k$ in its closure for each $k\in K$.
Then $\pi |_{\cH_I}$ is disjoint from $\pi_0\oplus \bigoplus_{\rho\in E} \rho^{\oplus \Nb}$,
and $\pi'$ is weakly contained in $\pi_0\oplus \bigoplus_{\rho\in E} \rho^{\oplus \Nb}$
\cite[Thm.\ 3.4.10]{Dix}
and hence lies in the orbit closure of every representation in $\Rep (A,\cH )$
unitarily conjugate to $\pi_0\oplus \bigoplus_{\rho\in E} \rho^{\oplus \Nb}$ by Lemma~\ref{L-wc}.
This finishes the proof.
\end{proof}

Write $\FE (A,\cH )$ for the $\cU (\cH )$-invariant set of faithful essential representations in\linebreak
$\Rep (A,\cH )$, which is a dense $G_\delta$ by Lemma~\ref{L-fe}.

\begin{theorem}\label{T-turbulent}
Let $A$ be a separable $C^*$-algebra.
If $\hat{A}_\iso$ is dense in $\hat{A}$ then the action of $\cU (\cH )$ on $\Rep (A,\cH )$ has a dense
$G_\delta$ orbit, while if $\hat{A}_\iso$ is not dense in $\hat{A}$ then the restriction of the
action to $\FE (A,\cH )$ is turbulent.
Furthermore, the action on $\Rep (A,\cH )$ is turbulent precisely when $A$ is simple and
not isomorphic to the compact operators
on some Hilbert space.
\end{theorem}

\begin{proof}
Suppose first that $\hat{A}_\iso$ is dense in $\hat{A}$. Let $I$ be the closed ideal of $A$ such that
$\hat{I} = \hat{A}_\iso$.
For each $\pi\in\hat{A}_\iso$ let $I_\pi$ be the closed ideal of $A$ such that $\hat{I}_\pi = \{ \pi \}$.
Note that the set $V$
of all representations in $\Rep (A,\cH )$ which are nondegenerate on $I$ can be expressed as
\[ \bigcap_{n=1}^\infty \big\{ \sigma\in\Rep (G,\cH ) : \| \sigma (a)\xi_n \| > {\textstyle\frac12} \| \xi_n \|
\text{ for some } a \text{ in the unit ball of } A \big\} \]
for a given dense sequence $\{ \xi_n \}_{n=1}^\infty$ in the unit sphere of $\cH$ and hence is a $G_\delta$.
Given a $\pi\in\hat{I}$, choose a projection $p_\pi \in I_\pi$ such that $\pi (p_\pi )$ has rank one.
Then for every $n\in\Nb$ the set $V_{\pi ,n}$ of all $\pi\in\Rep (A,\cH )$ such that
$\Tr (\sigma (p_\pi )) > n$ is easily seen to be open by expressing
$\Tr$ in terms of a fixed orthonormal basis of $\cH$.
As $\hat{A}$ is second countable, $\hat{I}$ is countable.
Set $\rho = \bigoplus_{\pi\in\hat{A}_\iso} \pi^{\oplus\Nb}$.
By Lemma~1.4 of \cite{Wang}, every representation of $I$ is a direct sum of irreducible representations.
Consequently the orbit of $\rho$ is precisely
$V\cap\bigcap_{\pi\in\hat{A}_\iso} \bigcap_{n\in\Nb} V_{\pi ,n}$, a $G_\delta$ subset of $\Rep (A,\cH )$.
Moreover, since by Voiculescu's theorem the orbit closure of any element in $\Rep (A,\cH )$ is equal
to the orbit closure of some direct sum of irreducible representations \cite[Cor.\ 1.6]{NCWvNT},
we see by \cite[Thm.\ 3.4.10]{Dix} and Lemma~\ref{L-wc} that the orbit of $\pi$ is dense in $\Rep (A,\cH )$.

Suppose now that $\hat{A}_\iso$ is not dense in $\hat{A}$. To establish that the action of $\cU (\cH )$
on $\FE (A,\cH )$ is turbulent,
it suffices by Lemma~\ref{L-turbpoint} to show that every orbit in $\FE (A,\cH )$ is meager.
Let $\pi\in\FE (A,\cH )$. Let $I$ be closed ideal of $A$ such that $\hat{I}$
is the complement of the closure of $\hat{A}_\iso$ in $\hat{A}$.
Then $\pi (I) \neq \{ 0 \}$ by faithfulness, and so the orbit of $\pi$ is meager
by Lemma~\ref{L-not dense}.
Thus every orbit in $\FE (A,\cH )$ is meager and we have turbulence.

Finally, if $A$ is simple and not isomorphic to the compact operators on some Hilbert space then every
representation is faithful and essential and the topology on $\hat{A}$ is trivial. Thus by Lemma~\ref{L-turbpoint}
every orbit in $\Rep (A,\cH )$ is dense and so the action is turbulent in view of what we
know from above. If $A$ is isomorphic to the compact operators on some Hilbert space
then its spectrum is a singleton and so from above there is a dense $G_\delta$ orbit in $\Rep (A,\cH )$.
If $A$ is not simple then it has a nontrivial quotient and hence a
nonfaithful representation in $\Rep (A,\cH )$, and the orbit of this representation is nowhere
dense by Lemma~\ref{L-turbpoint}, so that the action fails to be turbulent.
\end{proof}

\begin{corollary}
Let $A$ be a separable unital antiliminary $C^*$-algebra.
Then the action of $\cU (\cH )$ on $\Rep (A,\cH )$ is generically turbulent.
\end{corollary}

\begin{proof}
This follows from the theorem because the existence of an isolated point of $\hat{A}$ would
yield an ideal isomorphic to the compact operators on a separable Hilbert space \cite[Lemma 1.3]{Wang},
contradicting antiliminarity.
\end{proof}

If $A$ is a separable $C^*$-algebra such that $\hat{A}$ is countable, then every representation
of $\Rep (A,\cH )$ is a direct sum of irreducible representations \cite[Lemma 1.4]{Wang}, and
the associated multiplicity function $\hat{A} \to \{ 0,1,\dots ,\infty \}$ is a complete
invariant for unitary equivalence. So in this case the classification of elements in $\Rep (A,\cH )$
up to unitary equivalence is smooth. On the other hand:

\begin{theorem}\label{T-not classifiable}
Let $A$ be a separable $C^*$-algebra such that $\hat{A}$ is uncountable. Then
the elements of $\Rep (A,\cH )$ up to unitary equivalence do not admit classification by
countable structures.
\end{theorem}

\begin{proof}
By the Cantor-Bendixson theorem, the set $P$ of condensation points in $\hat{A}$
is perfect (and in particular closed) and its complement is countable. Since $\hat{A}$ is
uncountable, $P$ is nonempty. Thus there is an ideal $I\subseteq A$ such that $P$
is equal to the set $\hat{A}_I$ of all $\sigma\in\hat{A}$ for which $\sigma (I) = 0$, and
the map $h : \hat{A}_I \to \widehat{A/I}$ obtained by passing to the quotient is a
homeomorphism \cite[Sect.\ 3.2]{Dix}. Then $A/I$ is nontrivial and $\widehat{A/I}$ contains
no isolated points. By Theorem~\ref{T-turbulent} above and Corollary~3.19 of \cite{COER}
the elements of of $\Rep (A/I,\cH )$ up to unitary equivalence do not admit classification by
countable structures. Since $\Rep (A/I,\cH )$ can be view as the closed set of all
representations in $\Rep (A,\cH )$ which factor through $A/I$, we obtain the theorem.
\end{proof}

Finally we turn to the problem of classifying irreducible representations. By Glimm's
theorem (see Section~6.8 of \cite{Ped}), a separable $C^*$-algebra is of type I
if and only if the Mackey Borel structure on $\hat{A}$ is standard.
This is strengthened by the following result, which was
shown by Hjorth for countable discrete groups \cite{NSIDGR} by a different argument
that can also be applied to our more general setting.

\begin{theorem}\label{T-irred classification}
Let $A$ be a separable non-type I $C^*$-algebra. Then the irreducible representations
of $A$ on $\cH$ up to unitary equivalence do not admit classification by countable structures.
\end{theorem}

\begin{proof}
Since $A$ is not of type I, by Glimm's theorem \cite[Thm.\ 6.8.7]{Ped}
there exists an essential irreducible representation
$\pi$ of $A$ on $\cH$.
Set $B = A/\ker (\pi )$. Write $\Irr (B,\cH )$ for the set of irreducible representations
in $\Rep (B,\cH )$, which is a dense $G_\delta$ by Lemma~\ref{L-turbpoint} and
\cite[Prop.\ 3.7.4]{Dix}.
Since $B$ admits a faithful irreducible representation
its spectrum $\hat{B}$ contains no isolated points and so by Theorem~\ref{T-turbulent} the
action of $\cU (\cH)$ on $\Rep (B,\cH )$, and hence also on $\Irr (B,\cH )$, is generically turbulent.
We may view $\Irr (B,\cH )$ as the closed subset of $\Irr (A,\cH )$ consisting of those irreducible
representations of $A$ which factor through $B$, and so we conclude by Corollary~3.19 of \cite{COER}
that the irreducible representations of $A$ do not admit classification by countable structures.
\end{proof}


\section{Representations of groups}\label{S-groups}

Here we record some consequences of Section~\ref{S-rep} for unitary group representations.
Let $G$ be a second countable locally compact group.
We write $\Rep (G,\cH )$ for $\Rep (C^* (G),\cH )$, and we denote by $\WM (G,\cH )$ the subset of
representations in $\Rep (G,\cH )$ which are weak mixing, i.e., which have no nonzero finite-dimensional
subrepresentations. The subset $\WM (G,\cH )$ is a $G_\delta$ \cite{BerRos} and
it is dense precisely when $G$ has property T \cite{KP,Kechris}.

As mentioned previously, Theorem~\ref{T-irred classification} specializes in the group setting
to the following theorem of Hjorth \cite{NSIDGR}.
Actually Hjorth proved the result in the discrete case but his argument works more generally using Glimm's theorem.

\begin{theorem}[Hjorth]
Suppose that $G$ is not of type I. Then the irreducible representations
of $G$ do not admit classification by countable structures.
\end{theorem}

By Theorem~\ref{T-not classifiable}, the orbit equivalence relation of $\cU (\cH )$ acting on $\Rep (G,\cH )$
either does not admit classification by countable structures or is smooth according to whether
$\hat{G}$ is uncountable or countable. As an example of Fell illustrates, it is possible
for a noncompact second countable locally compact group to have countable dual
(see Section~IV of \cite{Baggett}). On the other hand, we can deduce nonclassifiability for
countable structures when $G$ is countably infinite or a noncompact separable Lie group, and in these cases
we can furthermore restrict to representations that are weakly contained in $\lambda_G$, as we now explain.

We write $\hat{G}_\lambda$ for the reduced dual of $G$, i.e., the closed set of all elements in $\hat{G}$
which are weakly contained in $\lambda_G$. We denote by $\Rep_\lambda (G,\cH )$ the closed set of
representations in $\Rep (G,\cH )$ which are weakly contained in $\lambda_G$. By Proposition~H.2 in
\cite{Kechris} this is equal to the closure of the orbit of $\lambda_G^{\oplus\Nb}$ viewed as a
representation on $\cH$ via some unitary equivalence.
We write $\WM_\lambda (G,\cH )$ for $\WM (G,\cH ) \cap \Rep_\lambda (G,\cH )$.

The following is well known.

\begin{lemma}\label{L-no irred subrep}
Suppose that $G$ is countably infinite. Then the left regular
representation $\lambda_G$ has no irreducible subrepresentations and $\hat{G}_\lambda$ has no isolated points.
\end{lemma}

\begin{proof}
By Corollary~5.12 of \cite{Rieffel}, $G$ has no square-integrable irreducible representations
and thus, since square-integrability for a representation is equivalent to being a subrepresentation
of $\lambda_G$, $G$ has no irreducible subrepresentations. Consequently
$\hat{G}_\lambda$ has no isolated points by Corollary~1.9 of \cite{Wang}.
\end{proof}

\begin{theorem}\label{T-reg turbulent}
Suppose that $G$ is countably infinite. Then the action of $\cU (\cH )$ on\linebreak
$\Rep_\lambda (G,\cH )$ is generically turbulent. Furthermore, the action is turbulent precisely
when $C^*_\red (G)$ is simple.
\end{theorem}

\begin{proof}
The result then follows by Lemma~\ref{L-no irred subrep} and Theorem~\ref{T-turbulent}.
\end{proof}

Combining Theorem~\ref{T-reg turbulent} with Corollary~3.19 of \cite{COER} yields:

\begin{theorem}\label{T-reg nonclassifiable}
Suppose that $G$ is countably infinite. Then the elements of $\Rep_\lambda (G,\cH )$ up to unitary
equivalence do not admit classification by countable structures.
\end{theorem}

\begin{remark}\label{R-WM}
In Theorems~\ref{T-reg turbulent} and \ref{T-reg nonclassifiable} we can replace
$\Rep_\lambda (G,\cH )$ by $\WM_\lambda (G,\cH )$, as follows from the following fact,
which we record as a proposition for future reference.
\end{remark}

\begin{proposition}\label{P-WM}
For a second countable locally compact group $G$, $\WM_\lambda (G,\cH )$
is a dense $G_\delta$ subset of $\Rep_\lambda (G,\cH )$.
\end{proposition}

\begin{proof}
If $G$ is amenable then $\Rep_\lambda (G,\cH ) = \Rep (G,\cH )$ and, by Theorem~2.5 of \cite{BerRos},
$\WM (G,\cH )$ is a dense $G_\delta$ subset of $\Rep (G,\cH )$.
If $G$ is nonamenable then every element of $\Rep_\lambda (G,\cH )$ is weakly mixing, for
if an element of $\Rep_\lambda (G,\cH )$ contains a finite-dimensional subrepresentation $\pi$
then $\pi\otimes\bar{\pi}$ contains the trivial representation and is weakly contained in $\lambda_G$,
contradicting nonamenability.
\end{proof}

Theorem~2.5 of \cite{Baggett} asserts that a separable Lie group whose reduced dual is countable
must be compact, and so by Theorem~\ref{T-not classifiable} we can conclude the following.

\begin{theorem}\label{T-Lie nonclassifiable}
Let $G$ be a separable noncompact Lie group. Then the elements of\linebreak
$\Rep_\lambda (G,\cH )$ up to unitary
equivalence do not admit classification by countable structures.
\end{theorem}


\section{Trace-preserving actions}\label{S-nonclassifiability}

Let $M$ be a von Neumann algebra with separable predual, and let $\tau$ be a faithful normal tracial state
on $M$. We write $\| \!\cdot \!\|_2$ for the $\tau$-norm on $M$, i.e., $\| a \|_2 = \tau (a^* a)^{1/2}$.
Let $G$ be a second countable locally compact group.
We denote by $\Act (G,M,\tau )$ the Polish space of continuous $\tau$-preserving actions of $G$ on $M$
whose topology has as a basis the sets
\[ Y_{\alpha ,K , \Omega ,\varepsilon} =
\{ \beta\in\Act (G,M,\tau ) : \| \beta_s (a) - \alpha_s (a) \|_2 < \varepsilon
\text{ for all } s\in K \text{ and } a\in\Omega \} \]
where $\alpha\in\Act (G,M,\tau )$, $K$ is a compact subset of $G$, $\Omega$ is a
finite subset of $M$, and $\varepsilon > 0$.
We equip the group $\Aut (M,\tau )$ of $\tau$-preserving automorphisms of $M$ with the Polish topology
which has as a basis the sets
\[ Z_{\alpha , \Omega ,\varepsilon} =
\{ \beta\in\Aut (R,G) : \| \beta (a) - \alpha (a) \|_2 < \varepsilon
\text{ for all } a\in\Omega \} \]
where $\alpha\in\Aut (R,G)$, $\Omega$ is a finite subset of $M$, and $\varepsilon > 0$.
We have a continuous action of $\Aut (M,\tau )$ on $\Act (G,M,\tau )$ given by
$(\gamma\cdot\alpha )_s (a) = (\gamma\circ\alpha_s \circ\gamma^{-1} )(a)$
for all $s\in G$, $a\in A$, $\gamma\in\Aut (M,\tau )$, and $\alpha\in\Act (G,M,\tau )$.

For every $\alpha\in\Act (G,M,\tau )$ we write $\kappa_\alpha$ for the associated
the unitary representation of $G$ on the GNS Hilbert space $L^2 (M,\tau )$
given by $\pi_\sigma (s) a\xi = \alpha_s (a) \xi$ for $a\in M$, where $\xi$ is the canonical cyclic
vector and $M$ is viewed as acting on $L^2 (M,\tau )$ via left multiplication. The
restriction of $\kappa_\alpha$ to $L^2 (M,\tau )\ominus\Cb\unit$ will be denoted $\kappa_{\alpha ,0}$.

An action $\alpha\in\Act (G,M,\tau )$ is said to be {\it ergodic} if $\kappa_{\alpha ,0}$
is ergodic (i.e., if $\kappa_{\alpha ,0}$ has no nonzero $G$-invariant vectors),
and {\it weakly mixing} if $\kappa_{\alpha ,0}$ is weakly mixing (i.e.,
if $\kappa_{\alpha ,0}$ has no nonzero finite-dimensional subrepresentations).
See \cite[App.\ D]{Vaes} for some standard characterizations of weak mixing for actions.
We write $\WM (G,M,\tau )$ for the set of weakly mixing actions in $\Act (G,M,\tau )$. This
is a $G_\delta$ set, as the proof of Proposition~2.3 in \cite{BerRos} shows.

\subsection{Freeness}\label{SS-freeness}
An automorphism $\theta$ of a von Neumann algebra $M$ is said to be
{\em properly outer} if for every nonzero $\theta$-invariant projection $p$
the restriction of $\theta$ to $pMp$ is not inner \cite{Kallman}\cite[Defn.\ XVII.1.1]{Tak3}.
We may equivalently quantify over $\theta$-invariant projections in the
centre of $M$ (see the comment after Theorem~XVII.1.2 in \cite{Tak3}).
An action $\alpha$ of $G$ on $M$ is said to be {\em free} if
$\alpha_s$ is properly outer for every $s\in G\setminus \{ e \}$.
In the commutative case this is equivalent to
the usual definition for actions on probability spaces.
The aim of this subsection is to show that when $G$ is countable the free actions
form a $G_\delta$ subset of $\Act (G,M,\tau )$, which was observed by Glasner and King
for measure-preserving tranformations of a standard atomless probability space \cite{GK}.
To this end we will establish in Lemma~\ref{L-properly outer} some characterizations of
proper outerness in the tracial case.
Compare \cite[Thm.\ 1.2.1]{OCC} and \cite[Thm.\ 3.3]{SMASA}.

An example of a noncommutative free action in the case that $G$ is countably infinite
is the Bernoulli shift $\beta$ on the weak operator closure of $M_n^{\otimes G}$ (with $n\geq 2$)
in the tracial representation, which is isomorphic to $R$.
The freeness of $\beta$ can be seen as follows.
If $s$ is an element of $G$ of infinite order then it is well known and easy to check
that the automorphism $\beta_s$ is mixing and in particular ergodic, in which case $\beta_s$
is not inner, for otherwise any unitary witnessing the innerness would be different from $\unit$
and fixed by $\beta_s$. If $s$ is an element of $G\setminus \{ e \}$ of finite
order then $\beta_s$ is conjugate to an automorphism as in the statement of the following
proposition and hence is not inner. Thus $\beta$ is free.

\begin{proposition}\label{P-not inner}
Let $n\geq 2$ be an integer and let $\alpha$ be an automorphism of $M_n$ which
is not the identity. Let $\theta$ be the extension of
$\alpha^{\otimes\Nb}$ on the weak operator closure $R$ of $M_n^{\otimes\Nb}$
in the tracial representation. Then $\theta$ is not inner.
\end{proposition}

\begin{proof}
We can express $\alpha$ as $\Ad u$ for some unitary $u\in M_n$.
Suppose that $\theta$ is inner. Then $\beta=\Ad v$ for some unitary $v\in R$.
Let $k\geq 1$. Then $\theta$ restricts to $\Ad u^{\otimes [1,k]}$
on $M_n^{\otimes [1,k]}$. Denote by $E_k$ the trace-preserving conditional expectation
of $R$ onto $M_n^{\otimes [1,k]}$.
For every $a\in M_n^{\otimes [1,k]}$,
from $(u^{\otimes [1,k]})a(u^{\otimes [1,k]})^*v=\theta(a)v=va$ we obtain
$(u^{\otimes [1,k]})a(u^{\otimes [1,k]})^*E_k(v)=E_k(v)a$. Then
$(u^{\otimes [1,k]})^*E_k(v)$ is in the centre of
$M_n^{\otimes [1,k]}$ and hence $E_k(v)=\lambda_k(u^{\otimes [1,k]})$ for some $\lambda_k\in \Cb$.
As $k\to \infty$, $|\lambda_k|=\| \lambda_k(u^{\otimes [1,k]})\|_2= \| E_k(v)\|_2\to \| v\|_2=1$.
We also have
$E_k(v)=E_k(E_{k+1}(v))=E_k(\lambda_{k+1}(u^{\otimes [1,k+1]}))=
\tr(u)\lambda_{k+1}(u^{\otimes [1,k]})$,
where $\tr$ denotes the normalized trace on $M_n$. Thus $\lambda_k=\tr(u)\lambda_{k+1}$.
It follows that $|\tr(u)|=1$. By the Cauchy-Schwarz inequality we get $u\in \Cb$. This implies
that $\alpha=\Ad u$ is trivial, contradicting our assumption. Therefore $\beta$ is not inner.
\end{proof}

The equivalence of the first two conditions in the following lemma is due to
Connes, who proved it for arbitrary automorphisms
of a countably decomposable von Neumann algebra \cite[Thm.\ 1.2.1]{OCC}.

\begin{lemma}\label{L-properly outer}
Let $M$ be a von Neumann algebra with faithful normal tracial state $\tau$.
Let $\theta\in \Aut(M,\tau )$. Let $0<\lambda\leq 1/3$ and $0<\varepsilon<1$.
Let $S$ be a $\tau$-norm dense subset of the set of all nonzero projections in $M$.
Then the following are equivalent:
\begin{enumerate}
\item $\theta$ is properly outer,

\item for every nonzero projection $p\in M$ there is a nonzero projection
$q\in M$ with $q\leq p$ and $\| q\theta(q)\| <\varepsilon$,

\item for every nonzero projection $p\in M$ there is a nonzero projection
$q\in M$ with $q\leq p$ and $\| q\theta(q)\|_2 <\varepsilon \| q\|_2$,

\item for every $p\in S$ there is a projection
$q\in M$ with $q\leq p$, $\| q\theta(q)\|_2 <\varepsilon \| q\|_2$,
and $\tau(q)\geq \lambda \tau(p)$,

\item for every nonzero projection $p\in M$ there is an
$x\in pMp$ with $\| x-\theta(x)\|_2> (1-\varepsilon)\| x\|_2$,

\item for every nonzero $\theta$-invariant projection $p\in M$ there is an
$x\in pMp$ with $\| x-\theta(x)\|_2> (1-\varepsilon) \| x\|_2$.
\end{enumerate}
\end{lemma}

\begin{proof}
(1)$\Rightarrow$(2). This is part of \cite[Thm.\ 1.2.1]{OCC}.

(2)$\Rightarrow$(3). With $q$ as in (2) we have
\begin{align*}
\| q\theta(q)\|^2_2 &= \| \theta(q)q\|^2_2 =\tau(qq\theta(q)qq)\le
\tau(q \| q\theta(q)q\| q) \\
&= \| q\theta(q)q\| \tau(q)=\| q\theta(q)\|^2 \| q\|^2_2< \varepsilon^2 \| q\|^2_2.
\end{align*}

(3)$\Rightarrow$(5). This follows by observing that $q$ as in (3) satisfies
\[ \| q-\theta(q)\|_2 \geq \| q-q\theta(q)\|_2
\ge \| q\|_2-\| q\theta(q)\|_2 >(1-\varepsilon)\| q\|_2. \]

(5)$\Rightarrow$(6). Trivial.

(6)$\Rightarrow$(1). Suppose that $\theta$ is not properly outer. Then
there exists a nonzero $\theta$-invariant projection $e\in M$ such that $\theta|_{eMe}$ is
equal to $\Ad u$ for some unitary $u\in eMe$. Let $p$ be a nonzero spectral
projection of $u$ such that $\| up-tp\|<(1-\varepsilon)/2$ for some
$t\in \Cb$ with $|t|=1$. Then $p$ is $\theta$-invariant, and
for every $x\in pMp$ we have
\begin{align*}
\| x-\theta(x)\|_2 &= \| (tp)x(\bar{t}p)-upxpu^*\|_2 \\
&\leq \| (tp-up)x(\bar{t}p)\|_2+\| upx(\bar{t}p-pu^*)\|_2 \\
&\leq \| tp-up\| \| x\|_2 \| \bar{t}p\|+\| up\| \| x\|_2 \| \bar{t}p-pu^*\| \\
&\leq (1-\varepsilon) \| x\|_2,
\end{align*}
contradicting (6).

(3)$\Rightarrow$(4). Let $p\in S$. Let $\{q_j\}_{j\in J}$ be a maximal family of subprojections
of $p$ such that $q_j\perp q_k$ and $q_j\perp \theta(q_k)$ for all distinct $j, k\in J$
and $\| q_j\theta(q_j)\|_2<\varepsilon \| q_j\|_2$ for each $j\in J$.
Set $q=\sum_{j\in J}q_j$.
Then
\begin{align*}
\| q\theta(q)\|^2_2
&= \tau(q\theta(q)q)=\sum_j \tau(q_j\theta(q_j)q_j)
=\sum_{j\in J}\| q_j\theta(q_j)\|^2_2 \\
&< \varepsilon^2 \sum_{j\in J}\| q_j\|^2_2=\varepsilon^2 \| q\|^2_2,
\end{align*}
and hence $\| q\theta(q)\|_2< \varepsilon \| q\|_2$.
We claim that $p \precsim \theta^{-1}(q)\vee q\vee \theta(q)$. Suppose that this is not true.
Since the range projection $p'$ of $p(\theta^{-1}(q)\vee q\vee \theta(q))$ is equivalent
to the orthogonal complement of its kernel projection, which is a subprojection
of $\theta^{-1}(q)\vee q\vee \theta(q)$, we see that $p'\neq p$. By (3) we can find
a projection $e\in M$ with $e\leq p-p'$ and $\| e\theta(e)\|_2< \varepsilon \| e\|_2$.
Then $e(\theta^{-1}(q)\vee q\vee \theta(q))=ep(\theta^{-1}(q)\vee q\vee \theta(q))=0$
and hence we can add $e$ to the family $\{p_j\}_{j\in J}$ to get a larger family,
which is a contradiction. Thus $p \precsim \theta^{-1}(q)\vee q\vee \theta(q)$ as claimed.
Note that for any projections $e_1, e_2\in M$
we have $e_1\vee e_2=e_1+e_3$ for some $e_3\thicksim e_2-e_1\wedge e_2$
and hence $\tau(e_1\vee e_2)\leq \tau(e_1)+\tau(e_2)$. Thus
\[ \tau(p)\leq \tau(\theta^{-1}(q)\vee q\vee \theta(q))
\leq \tau(\theta^{-1}(q))+\tau(q)+\tau(\theta(q))=3\tau(q). \]

(4)$\Rightarrow$(5). With $\varepsilon$ as in (4), let us show that (5) holds
for any given $\varepsilon'$ satisfying $1>\varepsilon'>\varepsilon$.
Take a $p'\in S$ with $\| p-p'\|_2<\delta$ for $\delta$ to be determined later.
By (4) we can find a projection $q'\in M$ with $q'\leq p'$ and $\| q'\theta(q')\|_2 <\varepsilon \| q'\|_2$
and $\tau(q')\geq \lambda \tau(p')$. As in the proof of (3)$\Rightarrow$(5) we have
$\| q'-\theta(q')\|_2 >(1-\varepsilon)\| q'\|_2$.
Set $x=pq'p\in pMp$. Then
\[ \| x-q'\|_2=\| pq'p-p'q'p'\|_2 \le 2\| p-p'\|_2 <2\delta \]
and hence
\[ \| x\|_2 >\| q'\|_2-2\delta \ge \lambda^{1/2}\| p'\|_2-2\delta>
\lambda^{1/2}(\| p\|-\delta)-2\delta\ge \frac{4}{\varepsilon'-\varepsilon}\delta , \]
granted $\delta$ is chosen small enough. Therefore
\begin{align*}
\| x-\theta(x)\|_2
&\geq \| q'-\theta(q')\|_2-2\| x-q'\|_2 >(1-\varepsilon)\| q'\|_2 -4\delta \\
&\geq (1-\varepsilon)\| x\|_2 -4\delta
\geq (1-\varepsilon')\| x\|_2.
\end{align*}
\end{proof}

\begin{lemma}\label{L-free Gdelta}
Let $M$ be a von Neumann algebra with separable predual and faithful normal tracial state $\tau$.
Then for a countable $G$ the set of free actions in $\Act (G,M,\tau )$ is a $G_\delta$.
\end{lemma}

\begin{proof}
Take an increasing sequence $K_1 \subseteq K_2 \subseteq\dots$ of finite subsets of $G$ whose union
is $G\setminus \{ e \}$, and an increasing sequence $\Omega_1 \subseteq\Omega_2 \subseteq\dots$
of finite sets of nonzero projections in $M$ whose union is $\tau$-norm dense in the set of all nonzero
projections in $M$. For every $n\in\Nb$ write $W_n$ for the open set of all
$\alpha\in\Act (G,M,\tau )$ such that for every $s\in K_n$ and $p\in\Omega_n$ there is a projection
$q\leq p$ satisfying $\| q\alpha_s (q) \|_2 < \| q \|_2 /3$ and $\| q \|_2 \geq \| p \|_2 /2$.
Then $W := \bigcap_{n=1}^\infty W_n$ is a $G_\delta$, and it consists precisely of the free actions
by Lemma~\ref{L-properly outer}.
\end{proof}

\subsection{Nonclassifiability by countable structures}
The proof of the theorems in this subsection will require Gaussian and Bogoliubov actions
(cf.\ the proof of Theorem~13.7 in \cite{Kechris}),
and so we will begin by briefly recalling these constructions.

Gaussian Hilbert spaces provide a standard means for producing actions from representations
(see Appendix~E of \cite{Kechris} or Appendix~A.7 of \cite{BHV}).
Let $\cH$ be a separable infinite-dimensional real Hilbert space.
Write $\nu$ for the standard Gaussian measure $(2\pi )^{-1/2} e^{-x^2 /2} \, dx$ on $\Rb$.
Fix an isometric isomorphism $\varphi : \cH \to \cH^{:1:}$ where $\cH^{:1:}$ is
the closed subspace of $L_2 (\Rb^\Nb , \nu^\Nb )$ spanned by the coordinate projections, i.e.,
the first Wiener chaos.
Then associated to each orthogonal operator $S$ on $\cH$ is a unique automorphism in $\Aut (\Rb^\Nb ,\nu^\Nb )$
whose Koopman operator restricts on $\cH^{:1:}$ to $\varphi\circ S \circ\varphi^{-1}$.
Thus from every representation in $\Rep (G,\cH )$ we obtain a $\nu^\Nb$-preserving action of $G$
on $(\Rb^\Nb , \nu^\Nb )$, called a {\it Gaussian action}, and up to conjugacy this action depends
only on the orthogonal equivalence class of the representation. The orthogonal Koopman representation for 
the Gaussian action associated to a $\pi\in\Rep (G,\cH )$ is orthogonally equivalent to the direct sum
$\bigoplus_{n=0}^\infty S^n \pi$ of all symmetric tensor powers of $\pi$. 

By the Gaussian action associated to a unitary 
representation of $G$ on a separable infinite-dimensional complex Hilbert space $\cH$ we 
mean the Gaussian action obtained from the induced orthogonal representation of $G$ on the realification
of $\cH$. From the description of the orthogonal Koopman representation in the latter case
we see that if $\pi$ is weakly mixing then so is
$\beta$, and if $\pi$ is weakly contained in $\lambda_G$ then so is $\kappa_{\beta ,0}$. These two facts
will be used in the proof of Theorem~\ref{T-discrete nonclassifiable}.

Bogoliubov actions are constructed as follows (see \cite{BraRob2} for a general reference).
Let $\cH$ be a separable infinite-dimensional Hilbert space.
The CAR algebra $A(\cH )$ is defined as the unique, up to $^*$-isomorphism, unital
$C^*$-algebra generated by operators $a(\xi )$ for $\xi\in\cH$ such that
the map $\xi\mapsto a(\xi )^*$ is linear and the anticommutation relations
\begin{align*}
a(\xi ) a(\zeta )^* + a(\zeta )^* a(\xi ) &= \langle \zeta ,\xi \rangle
\unit_{A(\cH )} ,\\
a(\xi ) a(\zeta ) + a(\zeta ) a(\xi ) &= 0,
\end{align*}
hold for all $\xi ,\zeta \in\cH$.
The $C^*$-algebra $A(\cH )$ is $^*$-isomorphic to the type $2^\infty$ UHF algebra and has
a unique tracial state \cite{BraRob2}.
The weak operator closure (equivalently, strong operator
closure) of $A(\cH )$ in the tracial representation is
isomorphic to the hyperfinite II$_1$ factor \cite[Prop.\ III.3.4.6]{OA} and we will
write it as $W(\cH )$.
Corresponding to a unitary operator $U$ on $\cH$ is
the Bogoliubov automorphism of $A(\cH )$ determined by
$a(\xi ) \mapsto a(U\xi )$ for $\xi\in\cH$.
A representation
$\pi\in\Rep (G,\cH )$ gives rise via Bogoliubov automorphisms to a continuous action of $G$ on
$A(\cH )$. Every continuous action of
$G$ on the CAR algebra by $^*$-automorphisms is trace-preserving and hence extends to a
continuous action on $W(\cH )$. We refer to these actions as {\em Bogoliubov actions}.
The tracial state $\tau$ on $A(\cH )$ is determined by
\[ \tau (a^* (\xi_1 )\cdots a^* (\xi_n ) a(\zeta_m ) \cdots a(\zeta_1 )) =
\delta_{nm} 2^{-n} \det ( \langle \xi_i , \zeta_j \rangle )_{i,j} .\]

The tracial GNS representation $(\pi_\tau , \cH_\tau , \Omega_\tau )$ and the corresponding
unitary implementation on $\cH_\tau$ of a unitary representation $\sigma$ of $G$ on $\cH$
can be described as follows (see Section~2 of \cite{Evans} for more details and references).
The antisymmetric Fock space $F(\cH )$ is defined as  $\bigoplus_{n=0}^\infty \Lambda^n \cH$ where
$\Lambda^n \cH$ is the $n$th antisymmetric tensor power of $\cH$, which for $n=0$ is defined as $\Cb$
with unit vector $\Omega$.
We have an irreducible representation $\rho$ of the CAR algebra on $F(\cH )$ determined by
\[ \rho (a(\xi )^* )\zeta_1 \wedge\cdots\wedge\zeta_n = \xi\wedge\zeta_1 \wedge\cdots\wedge\zeta_n . \]
For a unitary operator $U$ on $\cH$ we write $F(U)$ for the unitary operator on $F(\cH )$
which acts on $\Lambda^n \cH$ as $U^{\otimes n}$ and as the identity on $\Omega$.
Writing $\overline{\cH}$ for the conjugate Hilbert space of $\cH$ and
$\xi\mapsto\bar{\xi}$ for the canonical antilinear isometry from $\cH$ to $\overline{\cH}$,
we define a representation $\pi$ of $A(\cH )$ on $F(\cH )\otimes_2 F(\overline{\cH} )$ by
\[ \pi (a(\xi )) = \frac{1}{\sqrt{2}} \big[ \rho (a(\xi ))\otimes F(-\unit ) +
\unit\otimes\rho (a(\bar{\xi} )^* )\big] . \]
It can then be checked that the vector state on $A(\cH )$ associated to $\Omega\otimes\Omega$
coincides with $\tau$, so that $\pi_\tau$ can be identified with a subrepresentation of $\pi$. Writing $\beta$
for the Bogoliubov action associated to $\sigma$ and $F(\sigma )$ for the
unitary representation $s \mapsto F(\sigma_s )$ of $G$ on $F(\cH )$,
we observe accordingly that the representation $\kappa_{\beta ,0}$ of $G$ on $\cH_\tau \ominus\Cb\Omega_\tau$
arising from $\beta$ can be viewed as a subrepresentation of
\[ (F(\sigma )_0 \otimes\unit_G ) \oplus (F(\sigma )_0 \otimes F(\bar{\sigma} )_0 ) \oplus
(\unit_G \otimes F(\bar{\sigma} )_0 ) \]
where $F(\pi )_0$ for a representation $\pi$ means the restriction of $F(\pi )$ to the orthogonal complement
of $\Cb\Omega$. Since a tensor product of two representations is weakly contained in a tensor product
of two other representations under the assumption of factorwise weak containment, we see that
if $\sigma$ is weakly contained in $\lambda_G$ then so are $F(\sigma )_0$
and $F(\bar{\sigma} )_0$ and hence so is $\kappa_{\beta ,0}$. Since weak mixing for representations
is preserved under taking tensor products, we also observe that if $\sigma$ is weakly mixing
then $\beta$ is weakly mixing (as can alternatively be seen using the formula for $\tau$ from above).
Note also from the formula for $\tau$ that $\pi$ embeds as a subrepresentation of $\kappa_\beta$
via the map $\xi\mapsto \pi_\tau (\sqrt{2} a(\xi )^* )\Omega_\tau$ from $\cH$ to $\cH_\tau$.
We will invoke these facts in the proof of Theorem~\ref{T-discrete nonclassifiable}.

In the following results $M$ is assumed to be either $L_\infty (X,\tau )$ for some standard atomless probability
space $(X,\tau )$ or the hyperfinite II$_1$ factor $R$ with its unique faithful normal tracial state $\tau$.
We write $\Act_\lambda (G,M,\tau )$ for the closed set of all $\alpha\in\Act (G,M,\tau )$ such that the associated
representation $\kappa_{\alpha ,0}$ is weakly contained in the left regular representation.
We also define
\[ \WM_\lambda (G,M,\tau ) = \WM (G,M,\tau ) \cap \Act_\lambda (G,M,\tau ) , \]
which is a $G_\delta$ in $\Act (G,M,\tau )$,
and write $\FWM_\lambda (G,M,\tau )$ for the set of free weakly mixing actions in
$\Act_\lambda (G,M,\tau )$, which for countable $G$ is a $G_\delta$
by Lemma~\ref{L-free Gdelta}.

To establish the following theorem we argue by contradiction following the
scheme of the proof of Theorem~13.7 in \cite{Kechris}.

\begin{theorem}\label{T-discrete nonclassifiable}
Suppose that $G$ is countably infinite.
Then up to conjugacy the elements of $\FWM_\lambda (G,M,\tau )$
do not admit classification by countable structures.
\end{theorem}

\begin{proof}
Suppose that there does exist a Borel function $F$ from $\FWM_\lambda (G,M,\tau )$ to the space $\Theta_L$
of countable structures on some countable language $L$ such that $\alpha$ is conjugate to
$\beta$ if and only if $F(\alpha )\cong F(\beta )$. By Theorem~\ref{T-reg turbulent} and
Proposition~\ref{P-WM} there is a dense $G_\delta$ set
$Z$ of weakly mixing representations in $\Rep_\lambda (G,\cH )$ on which the action of $\cU (\cH )$ is turbulent.
Fix an action $\sigma$ in $\Act_\lambda (G,M,\tau )$ which is free and
weakly mixing, such as a Bernoulli shift (see Subsection~\ref{SS-freeness}). Then for every
$\gamma\in\WM (G,M,\tau )$ the action $\sigma\otimes\gamma$ is weakly mixing and, by Corollary~1.12 of
\cite{Kallman}, free.

Recall from the discussion above that if
$\pi$ is a weakly mixing representation in $\Rep_\lambda (G,\cH )$ then
the associated Gaussian action (if $M=L_\infty (X,\tau )$) or the associated Bogoliubov action
(if $M=R$) is contained in $\WM_\lambda (G,M,\tau )$.
By Theorem~3.18 of \cite{COER} there exists a $K\in \Theta_L$ such that $F(\sigma \otimes\beta_\pi ) \cong K$
for all $\pi$ in a comeager subset $Z_0$ of $Z$, where $\beta_\pi$ is the associated Gaussian or
Bogoliubov action and $\sigma \otimes\beta_\pi$ is viewed as an action on $M$ via some fixed isomorphism.
Thus there
is an $\alpha\in\WM_\lambda (G,M,\tau )$ such that $\alpha$ is conjugate to $\sigma\otimes\beta_\pi$ for
all $\pi\in Z_0$. But then $\kappa_\alpha$ contains every representation in $Z_0$,
which is a contradiction because the set of representations in $\Rep_\lambda (G,\cH )$ which are disjoint from
$\kappa_\alpha$ is a dense $G_\delta$ by Lemmas~\ref{L-not dense} and \ref{L-no irred subrep}
and hence has nonempty intersection with $Z_0$.
\end{proof}

The techniques of Hjorth using irreducible
representations \cite{NSIDGR}\cite[Thm.\ 13.7]{Kechris} can be applied to give
nonclassifiability-by-countable-structures results for actions of any second countable locally compact group
which is not compact and not type I. This excludes however many groups of interest,
such as $\Rb^d$ and $\SL (2,\Rb )$. The following theorem applies in particular to all
noncompact separable Lie groups which are not amenable \cite[Thm.\ 2.5]{Baggett} as well as many, if not
all, which are amenable.
To put the hypotheses into perspective, we remark as in Section~\ref{S-groups} that
Fell constructed an example of a noncompact second countable locally compact group
with countable dual (see Section~IV of \cite{Baggett}).

For a closed set $C\subseteq \hat{G}$, we write $\Rep_C (G,\cH )$ for the closed set of
representations in $\Rep (G,\cH )$ which are weakly contained in $C$, and
$\WM_C (G,\cH )$ for $\WM (G,\cH ) \cap \Rep_C (G,\cH )$.

\begin{theorem}\label{T-nondiscrete}
Let $G$ be a second countable locally compact group such that either (i) $G$ is
not amenable and $\hat{G}_\lambda$ is uncountable, or (ii) $G$ is amenable and
the set of isolated points in $\hat{G}$ is not dense.
Then up to conjugacy the elements of $\WM_\lambda (G,M,\tau )$ do not admit classification
by countable structures.
\end{theorem}

\begin{proof}
Let $\cH$ be a separable infinite-dimensional Hilbert space.
Suppose first that $G$ is not amenable and $\hat{G}_\lambda$ is uncountable.
Let $C$ be the closed set of condensation points of $\hat{G}_\lambda$, which is nonempty by the uncountability of
the latter. Since $G$ is not amenable $\hat{G}_\lambda$ contains no finite-dimensional representations, and so
$\WM_C (G,\cH ) = \Rep_C (G,\cH )$.
By Theorem~\ref{T-turbulent} there is a dense $G_\delta$ set
$Z\subseteq\WM_C (G,\cH )$ on which the action of $\cU (\cH )$ is turbulent.
Suppose that there exists a Borel function $F$ from $\WM_\lambda (G,M,\tau )$
to the space $\Theta_L$ of countable structures on some countable language $L$ such that $\alpha$ is conjugate to
$\beta$ if and only if $F(\alpha )\cong F(\beta )$. Theorem~3.18 of \cite{COER} yields the existence of a $K\in \Theta_L$
such that $F(\beta_\pi ) \cong K$ for all
$\pi$ in a comeager subset $Z_0$ of $Z$, where $\beta_\pi$ is the associated Gaussian action
if $M=L_\infty (X,\tau )$ or the associated Bogoliubov action on
$W(\cH )$ if $M=R$ (which, as in the proof of Theorem~\ref{T-discrete nonclassifiable}, are both
weakly mixing since $\pi$ is weakly mixing)
and $\beta_\pi$ is viewed as an action on $M$ via some fixed isomorphism.
Consequently there is an $\alpha\in\WM_\lambda (G,M,\tau )$ such that $\alpha$ is conjugate to $\beta_\pi$ for
all $\pi\in Z_0$. Then $\kappa_\alpha$ contains every representation in $Z_0$,
which is a contradiction because the set of representations in $\WM_C (G,\cH )$ which are disjoint from
$\kappa_\alpha$ is a dense $G_\delta$ by Lemma~\ref{L-not dense}
and hence has nonempty intersection with $Z_0$. Thus $\WM_C (G,M,\tau )$, and hence also
$\WM_\lambda (G,M,\tau )$, does not admit classification by countable structures.

Suppose now that $G$ is amenable and the set of isolated points in $\hat{G}$ is not dense. Then
$\WM (G,\cH ) = \WM_\lambda (G,\cH )$ and this set is a dense $G_\delta$ in $\Rep (G,\cH )$
\cite[Thm.\ 2.5]{BerRos}. Then by Theorem~\ref{T-turbulent} there is a dense $G_\delta$ set
$Z\subseteq\WM (G,\cH )$ on which the action of $\cU (\cH )$ is turbulent. Using this $Z$ we can now carry out
an argument by contradiction as in the first paragraph to obtain the desired conclusion.
\end{proof}

\begin{remark}
Conjugacy can be replaced by unitary equivalence in the statement of each of the
theorems in this subsection, as is clear from the proofs.
\end{remark}


\section{Turbulence and actions on the hyperfinite II$_1$ factor}\label{S-turb}

Our main goal in this section is to establish turbulence in the space of free actions of
a countably infinite amenable group on the hyperfinite II$_1$ factor $R$.
In Subsection~\ref{SS-ed} we show the meagerness of orbits in the space of actions of a countably infinite
amenable group on $R$ by developing a noncommutative version of an entropy and disjointness
argument of Foreman and Weiss \cite{FW}.
In Subsection~\ref{SS-turb} we show for general second countable locally compact $G$ how to construct
a $G$-action on $R$ which has dense orbit and is a turbulent point, and then deduce
turbulence in the space of free actions when $G$ is countably infinite and amenable by applying a theorem
of Ocneanu.

As usual, $G$ is assumed to be a second countable locally compact group,
subject to extra hypotheses as required.
In this section $\tau$ will invariably denote the unique normal tracial state on $R$.
We write $\Aut (R)$ for the automorphism group of $R$ and
$\Act (G,R)$ for the set of continuous actions of $G$ on $R$.
We regard these as Polish spaces under the topology defined in Section~\ref{S-nonclassifiability},
where we have dropped $\tau$ in the notation since every action on $R$ is $\tau$-preserving.
The set of free actions in $\Act (G,R)$ will be written $\Fr (G,R)$. We will require
the following two lemmas relating to freeness.

\begin{lemma}\label{L-free dense}
Suppose that $G$ is countable. Then $\Fr (G,R )$ is a dense $G_\delta$ subset of $\Act (G,R )$.
\end{lemma}

\begin{proof}
By Lemma~\ref{L-free Gdelta} we need only show the density.
Find a free action $\beta$ in $\Act (G,R )$.
If $G$ is infinite we may take a Bernoulli shift (see the first part of Subsection~\ref{SS-freeness}),
while if $G$ is finite then we may take an embedding $\varphi$ of $G$ into the unitary group
of $M_n$ for some $n$ and then use $s\mapsto (\Ad \varphi(s))^{\otimes \Nb}$ by Proposition~\ref{P-not inner}.
Take a dense sequence $\{ \alpha_i \}_{i=1}^\infty$ in
$\Act (G,R)$ and let $\alpha$ be an element of $\Act (G,R)$ which is conjugate to
$\beta\otimes\bigotimes_{i=1}^\infty \alpha_i \in
\Act (G,R\overline{\otimes} R^{\overline{\otimes}\Nb} )$.
Then the orbit of $\alpha$ is dense in $\Act (G,R)$, as can be established using an
argument as in the proof of Lemma~3.5 in \cite{KP}. Moreover, $\alpha$ is free since the tensor
product of a free action with any other action is free \cite[Cor.\ 1.12]{Kallman}, completing the proof.
\end{proof}

\begin{lemma}\label{L-free dense orbit}
Suppose that $G$ is countable and amenable.
Then every free action in $\Act (G,R )$ has dense orbit.
\end{lemma}

\begin{proof}
A theorem of Ocneanu shows that if $G$ is a countable amenable group then
given free actions $\alpha ,\beta\in\Act (G,R)$, a finite set $K\subseteq G$,
and an $\varepsilon > 0$ there exist a $\theta\in\Aut (R)$ and unitaries $u_s \in R$
for $s\in K$ such that, for every $s\in K$, $\| u_s - 1 \|_2 < \varepsilon$
and $\alpha_s = \theta\circ (\Ad u_s )\circ\beta\circ\theta^{-1}$
(see Section~1.4 of \cite{Ocneanu}). Consequently the orbit of
every free action is dense in $\Fr (G,R)$, and so by
Lemma~\ref{L-free dense} we obtain the result.
\end{proof}

\subsection{Entropy, disjointness, and meagerness of orbits}\label{SS-ed}

We will use the entropy of Connes-Narnhofer-Thirring \cite{CNT} as applied to actions of
discrete amenable groups on $R$. For a general reference on CNT entropy see \cite{NS}.
A {\it channel} is a u.c.p.\ (unital completely positive) map $\gamma : B\to R$ where $B$ is a
finite-dimensional $C^*$-algebra. Given channels $\gamma_i : B \to R$ for $i=1,\dots ,n$ we write
$H_\tau (\gamma_1 , \dots , \gamma_n )$ for the supremum of the entropies of the Abelian models
for $\gamma_1 , \dots , \gamma_n$ (see Section~III of \cite{CNT}). By Theorem~IV.3 of \cite{CNT}
the function $H_\tau (\cdot )$ is continuous with respect to the trace norm in the sense that, given
a finite-dimensional $C^*$-algebra $B$, for every $\varepsilon > 0$ there is a $\delta > 0$ such that,
for all $n\in\Nb$, if $\gamma_i , \gamma_i' : B\to R$ are channels with
$\sup_{\| a \| \leq 1} \| \gamma_i' (a) - \gamma_i (a) \|_2 < \delta$ for $i=1, \dots ,n$ then
\[ | H_\tau (\gamma_1' , \dots , \gamma_n' ) - H_\tau (\gamma_1 , \dots , \gamma_n ) | < n\varepsilon . \]

Suppose now that $G$ is discrete and amenable and let $\alpha\in\Act (G,R)$.
We define $h_\tau (\gamma , \alpha )$ as the limit of
\[ \frac{1}{|F|} H_\tau ((\alpha_s \circ\gamma )_{s\in F} ) \]
as $F$ becomes more and more invariant. This limit exists by subadditivity
\cite[Theorem 6.1]{LW}\cite[Proposition 3.22]{Ind}.
The CNT entropy $h_\tau (\alpha )$ of $\alpha$ is then defined as the supremum of $h_\tau (\gamma , \alpha )$
over all channels $\gamma : B\to R$.

\begin{lemma}\label{L-zero entropy}
Suppose that $G$ is countably infinite and amenable.
Then the set of $\alpha\in\Act (G,R)$ with $h_\tau (\alpha ) = 0$ is a dense $G_\delta$.
\end{lemma}

\begin{proof}
It suffices to show that, given an $\varepsilon > 0$, the set $S_\varepsilon$ of actions
in $\Act (G,R)$ with entropy at most $\varepsilon$ is a dense $G_\delta$, since
the collection of zero entropy actions is equal to $\bigcap_{n=1}^\infty S_{1/n}$.
Take a F{\o}lner sequence $\{ F_k \}_{k=1}^\infty$ in $G$ and finite-dimensional subfactors
$N_1 \subseteq N_2 \subseteq\dots$ of $R$ with trace norm dense union. Let
$\gamma_n : N_n \hookrightarrow R$ be the inclusion map for each $n\in\Nb$. Then
for all $n,m,l\in\Nb$ the set $S_\varepsilon (n,m,l)$ of all $\alpha\in\Act (G,R)$ such that
\[ H_\tau ((\alpha_s \circ\gamma_n )_{s\in F} ) < \Big( \varepsilon + \frac{1}{l} \Big)|F_k| \]
for some $k\geq m$ is open, and by the continuity properties of entropy $S_\varepsilon$ is equal to
$\bigcap_{l=1}^\infty \bigcap_{n=1}^\infty \bigcap_{m=1}^\infty S_\varepsilon (n,m,l)$, which is a $G_\delta$.
Kawahigashi showed in \cite{Kaw} that there exist free
actions realizing all possible nonzero values of entropy in the sense of Connes and St{\o}rmer, to which
CNT entropy specializes in this case \cite{NS}. Since the orbit of every free action is
dense in $\Act (G,R)$ by Lemma~\ref{L-free dense orbit}, we conclude that
$S_\varepsilon$ is dense.
\end{proof}

\begin{remark}
For any von Neumann algebra $M$ with separable predual and faithful
normal state $\tau$, the set of $\alpha\in\Act (G,M,\tau )$ with $h_\tau (\alpha ) = 0$ is a $G_\delta$.
To see this one can argue as in the proof of Lemma~\ref{L-zero entropy} by taking the sequence
$\{ \gamma_n \}_{n=1}^\infty$ to consist of a union of countable point-$\tau$-norm dense
sets of channels $B\to M$ over a collection of representatives $B$
of the countably many isomorphism classes of finite-dimensional $C^*$-algebras.
\end{remark}

The notions of joinings and disjointness \cite{ETJ} can be extended to noncommutative
actions as follows. A joining is a certain type of correspondence \cite[App.\ V.B]{NG}\cite{Cor}
and as such we need to consider opposite algebras in order to formulate the definition
as in the commutative case.
For a $C^*$-algebra $A$ we write $A^\op$ for the opposite $C^*$-algebra,
which has the same $^*$-linear structure as $A$ but with the multiplication reversed.
If $A$ is a von Neumann algebra then so is $A^\op$. For an element $b\in A$ we write
$\tilde{b}$ for the corresponding element of $A^\op$.
An action $\beta$ on $A$ gives rise to an action $\beta^\op$ on $A^\op$ defined
by $\beta^\op_s (\tilde{a} ) = \widetilde{\beta_s (a)}$ for all $s\in G$ and $a\in N$.

\begin{definition}\label{D-disjoint}
Let $M$ and $N$ be von Neumann algebras with faithful normal tracial states $\tau$ and
$\sigma$, respectively. Let $G$ a locally compact group. Let $\alpha\in\Act (M,G,\tau )$ and
$\beta\in\Act (N,G,\sigma )$. A {\em joining} of $\alpha$ and $\beta$ is
an $(\alpha\otimes\beta^\op )$-invariant state on the maximal $C^*$-tensor product
$M\otimes_\mx N^\op$ whose marginals are $\tau$ and $\sigma$.
We say that $\alpha$ and $\beta$ are {\em disjoint} if
$\tau\otimes\sigma$ is the only joining of $\alpha$ and $\beta$.
\end{definition}

Note that the definition of joining is symmetric in the sense that an $(\alpha\otimes\beta^\op )$-invariant
state on $M\otimes_\mx N^\op$ corresponds to an $(\beta\otimes\alpha^\op )$-invariant
state on $N\otimes_\mx M^\op$ via the canonical isomorphism of the latter with the opposite
$C^*$-algebra of $M\otimes_\mx N^\op$.

Specializing the picture for general correspondences between von Neumann algebras
\cite[App.\ V.B]{NG}\cite[{\S}1.2]{Cor},
we see that there is a bijective correspondence between the joinings of two actions
$\alpha\in\Act (M,G,\tau )$ and $\beta\in\Act (N,G,\sigma )$
and the $G$-equivariant unital complete positive maps $\varphi : M\to N$ such that $\tau = \sigma\circ\varphi$.
We associate to a joining $\omega$ of $\alpha$ with $\beta$
such a unital completely positive map $\varphi : M\to N$ as follows.
We define the bounded operator $S : L^2 (N,\sigma ) \to L^2 (M\otimes_\mx N^\op ,\omega )$
by setting $Sa = \unit_M \otimes \tilde{a}$ for $a\in N$, viewing all of these elements
as vectors in the appropriate GNS Hilbert space. We also define
the representation $\pi : M \to \cB (L^2 (M\otimes_\mx N^\op ,\omega ))$ by
$\pi (a) \zeta = (a\otimes\unit_{N^\op} )\zeta$, viewing $M\otimes\unit_{N^\op}$
as acting on $L^2 (M\otimes_\mx N^\op ,\omega )$. For $a\in M$
set $\varphi (a) = S^* \pi (a) S$. Then $S^* \pi (a) S$
commutes with right multiplication by elements of $N$ on $L^2 (N,\sigma )$, so that
$\varphi (a) \in (N')' = N$.

In the reverse direction, given a $G$-equivariant unital complete positive map $\varphi : M\to N$
such that $\tau = \sigma\circ\varphi$, we define on $\cH_0 = M\otimes N$ the sesquilinear form
\[ \langle a_1 \otimes b_1 , a_2 \otimes b_2 \rangle_\varphi =
\sigma (\varphi (a_2^* a_1 )b_1 b_2^* ) , \]
take the completion $\cH$ of $\cH_0$ modulo the null space of $\langle\cdot , \cdot\rangle_\varphi$,
and observe that the left and right actions of $M$ and $N$, respectively, pass to commuting
actions on $\cH$. This gives a representation of $M\otimes_\mx N^\op$, with $G$-invariant vector state
$\omega$ arising from the class of $\unit_M \otimes\unit_N$ in $\cH$ and having
$\tau$ and $\sigma$ as marginals. Note in particular that the identity map on $M$ gives
a joining of $\alpha$ with itself (the diagonal joining) which, as long as $M\neq\Cb$, is different
from the product joining. Thus every action in $\Act (G,M,\tau )$ for $M\neq\Cb$ is not disjoint from
any of its conjugates.

It is easily checked that, under the above correspondence,
the image of $\varphi$ is the scalars precisely
when it corresponds to the product state $\tau\otimes\sigma$
by the assumption on the marginals in the definition of joining.
This observation was used in \cite{MInd} to a give a linear-geometric proof of the
disjointness of zero entropy and completely positive entropy actions of discrete amenable groups
on a probability space. We will also apply this perspective in the proof of the following lemma,
only now using CNT entropy instead of $\ell_1$ geometry.

\begin{lemma}\label{L-entropy disjoint}
Suppose that $G$ is countably infinite and amenable. Let $\alpha$ be an action
in $\Act (G,R)$ with $h_\tau (\alpha ) = 0$. Then $\alpha$ is disjoint from every
action in $\Act (G,R)$ which is conjugate to the Bernoulli shift on the weak
operator closure of $M_2^{\otimes G}$.
\end{lemma}

\begin{proof}
We view $R$ as $(M_2 , \tr )^{\otimes G}$ and write $\beta$ for the Bernoulli shift on the
latter. Let $\gamma : B\to R$ be a channel. By Proposition~3.1.11 of \cite{NS}
we can find a channel $\gamma' : B\to R$
such that $|H_\tau (\gamma' ) - H_\tau (\gamma )| \leq  H_\tau (\gamma )/4$. and
$\gamma' (B) \subseteq M_2^{\otimes K}$ for some finite set $K\subseteq G$.
Given a nonempty finite set $F\subseteq G$, take a subset $F' \subseteq F$ which is maximal
with respect to the property that $s\notin KK^{-1} t$ for all distinct $s,t\in F'$.
Then $|F' | \geq |F|/|KK^{-1} |$ and
$H_\tau ((\beta_s \circ\gamma' )_{s\in F'} ) = |F' | H_\tau (\gamma' )$, and hence
\begin{align*}
H_\tau ((\beta_s \circ\gamma )_{s\in F'} )
&\geq H_\tau ((\beta_s \circ\gamma' )_{s\in F'} ) - \frac14 |F' | H_\tau (\gamma ) \\
&= |F' | \bigg( H_\tau (\gamma' ) - \frac14 H_\tau (\gamma )\bigg) \\
&\geq \frac12 |F' | H_\tau (\gamma )
\end{align*}
Now suppose we are given a $\tau$-preserving u.c.p.\ map $\varphi : R\to R$
such that $\varphi\circ\alpha_s = \beta_s \circ\varphi$ for all $s\in G$.
Then applying the above inequality to $\varphi\circ\gamma$ and
using the monotonicity properties of $H_\tau (\cdot )$ we have
\begin{align*}
\frac12 |F' | H_\tau (\varphi\circ\gamma ) &\leq H_\tau ((\beta_s \circ\varphi\circ\gamma )_{s\in F'} ) \\
&= H_\tau ((\varphi\circ\alpha_s \circ\gamma )_{s\in F'} ) \\
&\leq H_\tau ((\alpha_s \circ\gamma )_{s\in F'} )
\end{align*}
and so
\begin{align*}
\frac{1}{|F|} H_\tau ((\alpha_s \circ\gamma )_{s\in F} )
&\geq \frac{1}{|F' ||KK^{-1} |} H_\tau ((\alpha_s \circ\gamma )_{s\in F'} ) \\
&\geq \frac{1}{2|KK^{-1} |} H_\tau (\varphi\circ\gamma ) .
\end{align*}
But $\frac{1}{|F|} H_\tau ((\alpha_s \circ\gamma )_{s\in F} ) \to 0$ as $F$ becomes
more and more invariant since $\alpha$ has zero entropy. Therefore $H_\tau (\varphi\circ\gamma ) = 0$,
and so $\varphi\circ\gamma$ maps into the scalars \cite[Lemma~3.1.4]{NS}.
It follows that $\varphi$ is the map $a\mapsto \tau (a)\unit_R$, yielding the lemma.
\end{proof}

The following generalizes a result of del Junco from the commutative case \cite{dJ}.

\begin{lemma}\label{L-disjoint}
Let $M$ be a von Neumann algebra with separable predual and faithful normal tracial
state $\tau$. Let $\alpha\in\Act (G,M,\tau )$. Then the set $\alpha^\perp$ of
actions in $\Act (G,M,\tau )$ that are disjoint from $\alpha$ is a $G_\delta$.
\end{lemma}

\begin{proof}
For a compact set $K\subseteq G$, finite sets $\Omega ,\Theta\subseteq M$ and $\Upsilon\subseteq M_*$,
and $\varepsilon > 0$ we write $S(K,\Omega ,\Theta ,\Upsilon , \varepsilon )$ for the
set of all $\beta\in\Act (G,M,\tau )$ for which there exists a $\delta > 0$ such that
if $\varphi : M\to M$ is a $\tau$-preserving u.c.p.\ map satisfying
$\| \alpha_s (\varphi (a)) - \varphi (\beta_s (a)) \|_2 < \delta$ for all $s\in K$ and
$a\in\Omega$ then $| \omega (\varphi (a)) - \tau (a)\omega (\unit ) | < \varepsilon$ for all
$a\in\Theta$ and $\omega\in\Upsilon$. It is readily checked that
$S(K,\Omega ,\Theta ,\Upsilon , \varepsilon )$ is open. Take an increasing
sequence $K_1 \subseteq K_2 \subseteq\cdots$ of compact subsets of $G$ such that
$\bigcup_{n=1}^\infty K_n$ is dense in $G$, and increasing sequences
$\Omega_1 \subseteq\Omega_2 \subseteq\cdots$ and
$\Upsilon_1 \subseteq\Upsilon_2 \subseteq\cdots$ of finite subsets of $M$ and $M_*$,
respectively, such that $\bigcup_{n=1}^\infty \Omega_n$
is $\tau$-norm dense in $M$ and $\bigcup_{n=1}^\infty \Upsilon_n$ is norm dense in $M_*$.
Then $S := \bigcap_{n=1}^\infty \bigcup_{m=1}^\infty S(K_m , \Omega_m , \Omega_n ,\Upsilon_n , 1/n)$
is a $G_\delta$, and to complete the proof we will show that it is equal to $\alpha^\perp$.

Clearly $S\subseteq\alpha^\perp$. So let $\beta\in\Act(G,M,\tau ) \setminus S$ and let us show that
$\beta\notin\alpha^\perp$. For some $n\in\Nb$ we have
$\beta\notin \bigcup_{m=1}^\infty S(K_m , \Omega_m , \Omega_n , \Upsilon_n , 1/n)$. Then for every
$m\in\Nb$ there is a $\tau$-preserving u.c.p.\ map $\varphi_m : M\to M$ such that
$\| \alpha_s (\varphi_m (a)) - \varphi_m (\beta_s (a)) \|_2 < 1/m$ for all $s\in K_m$ and
$a\in\Omega_m$ and
$\sup_{a\in\Omega_n ,\omega\in\Upsilon_n} | \omega (\varphi_m (a)) - \tau (a)\omega (\unit ) | \geq 1/n$.
We can then find
a subsequence $\{ \varphi_{m_k} \}_{k=1}^\infty$, an $a_0 \in\Omega_n$, and an $\omega_0 \in\Upsilon_n$
such that $| \omega_0 (\varphi_{m_k} (a_0 )) - \tau (a_0 )\omega_0 (\unit ) | \geq 1/n$ for all $k$.
Now take a point-weak$^*$ limit point $\varphi$ of $\{ \varphi_{m_k} \}_{k=1}^\infty$.
Then $\varphi$ is $\tau$-preserving, $\alpha_s \circ\varphi = \varphi\circ\beta_s$
for all $s\in G$, and $| \omega_0 (\varphi (a_0 )) - \tau (a_0 )\omega_0 (\unit ) | \geq 1/n$,
so that $\varphi$ defines a joining different from the product one. Thus $\beta\notin\alpha^\perp$, and so
$S = \alpha^\perp$, as desired.
\end{proof}

\begin{lemma}\label{L-disjoint zero entropy}
Suppose that $G$ is countably infinite and amenable. Let $\alpha$ be an action
in $\Aut (G,R)$ with $h_\tau (\alpha ) = 0$. Then $\alpha^\perp$ is a dense $G_\delta$ subset of $\Aut (G,R)$.
\end{lemma}

\begin{proof}
By Lemma~\ref{L-disjoint} it suffices to show that $\alpha^\perp$ is dense. As observed
in Subsection~\ref{SS-freeness}, the Bernoulli shift on the tracial weak operator closure of
$M_2^{\otimes G}$ is free, and so any action in $\Act (G,R)$ that is conjugate to it has
dense orbit by Lemma~\ref{L-free dense orbit}. Thus $\alpha^\perp$ is dense
by Lemma~\ref{L-entropy disjoint}.
\end{proof}

\begin{lemma}\label{L-meager}
Suppose that $G$ is countably infinite and amenable. Then every orbit in $\Act (G,R)$ is meager.
\end{lemma}

\begin{proof}
Let $\alpha\in\Act (G,R)$ and suppose that the orbit is not meager. Then $h_\tau (\alpha ) = 0$
by Lemma~\ref{L-zero entropy}. But then $\alpha^\perp$ is a dense $G_\delta$ by
\ref{L-disjoint zero entropy} and every action in $\alpha^\perp$ is disjoint from every action in the orbit of
$\alpha$, yielding a contradiction.
\end{proof}

\subsection{Turbulence}\label{SS-turb}

All infinite tensor products in what follows are with respect to the trace.

\begin{lemma}\label{L-turbulent point}
Let $\{ \alpha_i \}_{i=1}^\infty$ be a dense sequence in
$\Act (G,R)$, and let $\alpha$ be an element of $\Act (G,R)$ which is conjugate to
$\bigotimes_{i=1}^\infty (\alpha_i \otimes\id_R ) \in
\Act (G,(R\overline{\otimes} R)^{\overline{\otimes}\Nb} )$.
Then $\alpha$ is a turbulent point and has dense orbit for the
action of $\Aut (R)$ on $\Act (G,R)$.
\end{lemma}

\begin{proof}
To show that $\alpha$ is a turbulent point,
let $Y$ be a neighbourhood of $\alpha$ in $\Act (G,R)$ and $Z$ a neighbourhood of
$\id_R$ in $\Aut (R)$ and let us demonstrate that the closure of the local orbit $\cO (\alpha ,Y,Z)$
has nonempty interior. By shrinking $Y$ and $Z$ if necessary we may suppose that
$Y = Y_{\alpha , K,\Omega ,\varepsilon}$ and $Z = Z_{\id ,\Omega ,\varepsilon}$ for some
compact set $K\subseteq G$, finite set $\Omega\subseteq R$, and $\varepsilon > 0$.

Let $\theta\in Y$. We will argue that $\theta$ is contained in
$\overline{\cO (\alpha ,Y,Z)}$. Let $L$ be a compact subset of $G$ containing $K$ and
$\Upsilon$ a finite subset of $R$ containing $\Omega$. Let $\delta > 0$
be such that $\| \theta_s (a) - \alpha_s (a) \|_2 < \varepsilon - \delta$
for all $a\in\Omega$ and $s\in K$.
For simplicity we will view $R$ as $(R\overline{\otimes} R)^{\overline{\otimes}\Nb}$ with
$\alpha$ acting as $\bigotimes_{i=1}^\infty (\alpha_i \otimes\id_R )$.
Take a $k\in\Nb$ large enough so that there is a finite-dimensional subfactor
$N\subseteq (R\overline{\otimes} R)^{\overline{\otimes} [1,k]}$ such that for every
$a\in\bigcup_{s\in L\cup\{ e\}} \theta_s (\Upsilon )$
we have $\| E(a) - a \|_2 < \delta /12$ where $E$ is the trace-preserving conditional
expectation from $(R\overline{\otimes} R)^{\overline{\otimes}\Nb}$ onto
$N\otimes\unit$, with $\unit$ denoting here and for the remainder of the paragraph
the unit in $(R\overline{\otimes} R)^{\overline{\otimes} [k+1,\infty ]}$.
For $a\in\Upsilon$ we write $a'$ for the element of $N$ such that $E(a) = a' \otimes\unit$.
Extend the embedding
$N\hookrightarrow (R\overline{\otimes} R)^{\overline{\otimes}\Nb}$ given by $a\mapsto a\otimes\unit$
to an isomorphism
$\Phi : (R\overline{\otimes} R)^{\overline{\otimes} [1,k]} \to (R\overline{\otimes} R)^{\overline{\otimes}\Nb}$
and define the action $\tilde{\theta}$ of $G$ on $(R\overline{\otimes} R)^{\overline{\otimes} [1,k]}$ by
$s\mapsto \Phi^{-1} \circ\theta_s \circ\Phi$.
For $a\in\Upsilon$ and $s\in L$ we have, noting that $\Phi (a' ) = E(a)$ and
$\Phi^{-1} (E(\theta_s (a)))\otimes\unit = E(\theta_s (a))$,
\begin{align*}
\| \theta_s (a) - \tilde{\theta}_s (a' )\otimes\unit \|_2
&= \| \theta_s (a) - \Phi^{-1} (\theta_s (E(a)))\otimes\unit \|_2 \\
&\leq \| \theta_s (a) - E(\theta_s (a)) \|_2 +
\| \Phi^{-1} (E(\theta_s (a)) - \theta_s (a))\otimes\unit \|_2 \\
&\hspace*{20mm} \ + \| \Phi^{-1} (\theta_s (a - E(a)))\otimes\unit \|_2 \\
&< \frac{\delta}{12} + \frac{\delta}{12} + \frac{\delta}{12} = \frac{\delta}{4} .
\end{align*}
Abbreviate $\bigotimes_{i=1}^k (\alpha_i \otimes\id_R )$
to $\sigma$ and fix an identification of $(R\overline{\otimes} R)^{\overline{\otimes} [1,k]}$ with $R$, so that
$\sigma$ and $\tilde{\theta}$ are regarded as actions on $R$.
By the density of the sequence $\{ \alpha_i \}_{i=1}^\infty$ we can find an integer
$l>k$ such that $\| \tilde{\theta}_s (a' ) - \alpha_{l,s} (a' ) \|_2 < \delta /4$ for
all $s\in L$ and $a\in\Upsilon$.
Note that, for all $s\in K$ and $a\in\Omega$,
\begin{align*}
\| \sigma_s (a' ) - \alpha_{l,s} (a' ) \|_2
&\leq \| \sigma_s (a' ) \otimes\unit - \alpha_s (a) \|_2 + \| \alpha_s (a) - \theta_s (a) \|_2 \\
&\hspace*{10mm} \ + \| \theta_s (a) - \tilde{\theta}_s (a' )\otimes\unit \|_2
+ \| \tilde{\theta}_s (a' ) - \alpha_{l,s} (a' ) \|_2 \\
&< \frac{\delta}{12} + \varepsilon - \delta + \frac{\delta}{4} + \frac{\delta}{4} \\
&= \varepsilon - \frac{5\delta}{12} .
\end{align*}

Take an integer $n > 64\varepsilon^{-2} \max \{ \| a \| : a\in\Omega \}$.
Since $R\cong M_n \otimes R$ and $\alpha$ contains infinitely many tensor product factors equal to $\id_R$,
we can view $\alpha$ as $\sigma\otimes\alpha_l \otimes\id_N \otimes\rho$
acting on $R\overline{\otimes} R\overline{\otimes} N\overline{\otimes} R$
where $N$ is a I$_n$ factor and $\rho$ is some automorphism of $R$.
Let $C$ be a commutative $n$-dimensional $^*$-subalgebra of $N$.
Let $e_1 , \dots , e_n$ be the minimal projections of $C$.
For a set $E \subseteq \{ 1,\dots , n \}$ we write $\gamma_E$ for the automorphism of
$R\overline{\otimes} R\otimes C$ which sends $\sum_{i=1}^n a_i \otimes e_i \in (R\overline{\otimes} R)\otimes C$ to
\[ \sum_{i\in E} \beta (a_i ) \otimes e_i + \sum_{i\in \{ 1,\dots ,n \} \setminus E} a_i \otimes e_i \]
where $\beta$ is the tensor product flip automorphism of $R\overline{\otimes} R$.
Then for every $a\in\Upsilon$ and $s\in L$ the image $a''$ of $a' \otimes\unit_{R \otimes C}$ under
$\gamma_E \cdot (\sigma_s \otimes\alpha_{l,s} \otimes\id_C )$ is equal to
\[ \sigma_s (a' ) \otimes\unit_R \otimes (1-p) + \alpha_{l,s} (a' ) \otimes\unit_R \otimes p \]
where $p$ is the projection $\sum_{i\in E} e_i$, and hence, for $a\in\Omega$ and $s\in K$,
\[ \| a'' - \sigma_s (a' )\otimes\unit_{R \otimes C} \|_2 \leq
\| \sigma_s (a' ) - \alpha_{l,s} (a' ) \|_2 \| p \|_2 < \varepsilon - \frac{5\delta}{12} . \]
Since every automorphism of the hyperfinite II$_1$ factor is approximately inner \cite[Thm.\ 2.16]{Tak3}
and $\gamma_E$ fixes the elements of $\unit_{R\overline{\otimes} R} \otimes C$,
we can find a unitary $u \in R\overline{\otimes} R\otimes C$ such that
$\| u b u^* - \gamma_E (b) \|_2 < \delta /8$
for all elements $b$ equal to
$((\sigma_s \otimes\alpha_{l,s} \otimes\id_C )\circ\gamma_E^{-1} )(a' \otimes\unit_{R \otimes C} )$
for some $s\in L\cup \{ e \}$ and $a\in\Upsilon$ or to $a' \otimes\unit_{R\otimes C}$ for some $a\in\Upsilon$.
Regarding $R\overline{\otimes} R\otimes C$ henceforth as a subalgebra of
$R\overline{\otimes} R \overline{\otimes} N \overline{\otimes} R$ under the embedding given
on elementary tensors by $a\otimes a\otimes c \mapsto a\otimes a\otimes c\otimes\unit_R$,
we set $\hat{\gamma}_E = \Ad u \in\Aut (R\overline{\otimes} R \overline{\otimes} N \overline{\otimes} R)$.
With an untagged $\unit$ denoting henceforth the unit of $R\overline{\otimes} N\overline{\otimes} R$,
for all $a\in\Upsilon$ we have, noting that $\gamma_E^{-1} (a' \otimes\unit )$ makes sense
as $a' \otimes\unit$ lies in the subalgebra of which $\gamma_E$ is an automorphism,
\[ \| \hat{\gamma}_E^{-1} (a' \otimes\unit ) - \gamma_E^{-1} (a' \otimes\unit ) \|_2
= \| \gamma_E (\gamma_E^{-1} (a' \otimes\unit )) - \hat{\gamma}_E (\gamma_E^{-1} (a' \otimes\unit )) \|_2
< \frac{\delta}{8} \]
and hence, writing $\tilde{\alpha}$ for the restriction $s\mapsto \sigma_s \otimes\alpha_{l,s} \otimes\id_C$
of $\alpha$ to $R\overline{\otimes} R\otimes C$, for $s\in L$ we have
\begin{align*}
\lefteqn{\| (\hat{\gamma}_E \cdot\alpha )_s (a' \otimes\unit ) -
(\gamma_E \cdot\tilde{\alpha} )_s (a' \otimes\unit ) \|_2} \hspace*{25mm} \\
&\leq \| \hat{\gamma}_E \circ\alpha_s (\hat{\gamma}_E^{-1} (a' \otimes\unit )
- \gamma_E^{-1} (a' \otimes\unit )) \|_2 \\
&\hspace*{10mm} \ + \| \hat{\gamma}_E (\alpha_s \circ\gamma_E^{-1} (a' \otimes\unit )) -
\gamma_E (\alpha_s \circ\gamma_E^{-1} (a' \otimes\unit ) \|_2 \\
&< \frac{\delta}{4} .
\end{align*}
Thus for $a\in\Omega$ and $s\in K$ we have
\begin{align*}
(\hat{\gamma}_E \cdot\alpha )_s (a) &\approx_{\delta /12} (\hat{\gamma}_E \cdot\alpha )_s (a' \otimes\unit ) \\
&\approx_{\delta /4} (\gamma_E \cdot\tilde{\alpha} )_s (a' \otimes\unit ) \\
&\approx_{\varepsilon - 5\delta /12} \sigma_s (a' ) \otimes\unit \\
&= \alpha_s (a' \otimes\unit ) \\
&\approx_{\delta /12} \alpha_s (a)
\end{align*}
so that $\hat{\gamma}_E \cdot\alpha \in Y$,
while in the case that $E = \{ 1,\dots , n\}$ we have, for $a\in\Upsilon$ and $s\in L$,
\begin{align*}
(\hat{\gamma}_E \cdot\alpha )_s (a)
&\approx_{\delta /12} (\hat{\gamma}_E \cdot\alpha )_s (a' \otimes\unit ) \\
&\approx_{\delta /4} (\gamma_{\{ 1,\dots ,n\}} \cdot\tilde{\alpha} )_s (a' \otimes\unit ) \\
&= \alpha_{l,s} (a' ) \otimes\unit \\
&\approx_{\delta /4} \tilde{\theta}_s (a' ) \otimes\unit  \\
&\approx_{\delta /4} \theta_s (a)
\end{align*}
so that $\hat{\gamma}_E \cdot\alpha \in Y_{\theta , L,\Upsilon ,\delta}$.

Note that, for all $i=1, \dots ,n$ and $a\in\Omega$,
\[ \| \gamma_{\{ i \}} (a' \otimes\unit ) - a' \otimes\unit \|_2
\leq \| \beta (a' ) - a' \|_2 \| e_i \|_2 < \frac{2\| a \|}{\sqrt{n}} \]
and thus
\begin{align*}
\| \hat{\gamma}_{\{ i \}} (a) - a \|_2 &\leq \| \hat{\gamma}_{\{ i \}} (a - a' \otimes\unit ) \|_2
+ \| \hat{\gamma}_{\{ i \}} (a' \otimes\unit ) - \gamma_{\{ i \}} (a' \otimes\unit ) \|_2 \\
&\hspace*{15mm} \ + \| \gamma_{\{ i \}} (a' \otimes\unit ) - a' \otimes\unit \|_2 +
\| a' \otimes\unit - a \|_2 \\
&< \frac{\varepsilon}{4} + \frac{\varepsilon}{4} + \frac{2\| a \|}{\sqrt{n}} + \frac{\varepsilon}{4} \\
&< \varepsilon ,
\end{align*}
so that $\hat{\gamma}_{\{ i \}} \in Z$.
From the previous paragraph, the action $(\hat{\gamma}_{\{ j \}} \circ\hat{\gamma}_{\{ j-1 \}}
\circ\cdots\circ\hat{\gamma}_{\{ 1 \}} )\cdot\alpha = \hat{\gamma}_{\{ 1,\dots , j \}} \cdot\alpha$
is contained in $Y$ for every $j = 1, \dots , n$, while
$(\hat{\gamma}_{\{ n \}} \circ\hat{\gamma}_{\{ n-1 \}} \circ\cdots\circ\hat{\gamma}_{\{ 1 \}} )\cdot\alpha
= \hat{\gamma}_{\{ 1,\dots ,n\}} \cdot\alpha$ is contained in $Y_{\theta , L,\Upsilon ,\delta}$.
We conclude that $\theta\in \overline{\cO (\alpha ,Y,Z)}$
and hence that $Y\subseteq \overline{\cO (\alpha ,Y,Z)}$, so that
$\overline{\cO (\alpha ,Y,Z)}$ has nonempty interior, as desired.

The density of the orbit of $\alpha$ now also follows, since
we showed that $Y\subseteq \overline{\cO (\alpha ,Y,Z)}$ for any $Y$ of the form
$Y_{\alpha , K,\Omega ,\varepsilon}$.
\end{proof}

The following is the analogue of Theorem~13.3 in \cite{Kechris}.

\begin{theorem}\label{T-meager gen turb}
The following conditions are equivalent:
\begin{enumerate}
\item every $\Aut (R)$-orbit in $\Act (G,R)$ is meager,

\item the action of $\Aut (R)$ on $\Act (G,R)$ is generically turbulent.
\end{enumerate}
Moreover if $G$ does not have property T then (1) and (2) are equivalent
to each of the following:
\begin{enumerate}
\item[(3)] every $\Aut (R)$-orbit in $\WM (G,R)$ is meager,

\item[(4)] the action of $\Aut (R)$ on $\WM (G,R)$ is generically turbulent.
\end{enumerate}
\end{theorem}

\begin{proof}
By \cite[Thm.\ 3.8]{KP}, the $G_\delta$ subset $\WM (G,R)$ of $\Act (G,R)$ is dense
precisely when $G$ does not have property T. Thus, since
meagerness of orbits is part of the definition of generic turbulence, it suffices to show
(1)$\Rightarrow$(2), and this follows by Lemma~\ref{L-turbulent point}.
\end{proof}

\begin{remark}
By Lemma~\ref{L-free dense}, when $G$ is countable and does not have property T
we can replace $\WM (G,R)$ by the set of free weakly mixing actions in $\Act (G,R)$
in the above theorem.
\end{remark}

\begin{lemma}\label{L-free turbulent}
Suppose that $G$ is countable and amenable. Let $\alpha$ be a free action of $G$ on $R$. Then
$\alpha$ is a turbulent point for the action of $\Aut(R)$ on $\Act(G, R)$.
\end{lemma}

\begin{proof}
By Lemma~\ref{L-turbulent point} there exists an action $\beta$ of $G$ on $R$ such
that for any finite set $K\subseteq G$ and $\Omega\subseteq R$ and
$\varepsilon>0$, the neighbourhood $Y = Y_{\beta, K, \Omega, \varepsilon}$ of $\beta$ is
contained in $\overline{\cO(\beta, Y, Z)}$, where $Z=Z_{\id, \Omega, \varepsilon}$.
Note that $G$ has a free action on $R$ by Lemma~\ref{L-free dense},
and so we may assume that $\beta$ is free by taking $\alpha_1$ in Lemma~\ref{L-turbulent point}
to be free, since the tensor product of a free action
with any other action is free \cite[Cor.\ 1.12]{Kallman}.
For a $\gamma \in \Aut(R)$ we have
$\gamma Y_{\beta, K, \Omega, \varepsilon}=Y_{\gamma\cdot\beta, K, \gamma\Omega , \varepsilon}$
and
$\gamma Z_{\id, \Omega, \varepsilon} \gamma^{-1}=Z_{\id, \gamma\Omega , \varepsilon}$,
and hence $\gamma \cO(\beta, Y_{\beta, K, \Omega, \varepsilon}, Z_{\id, \Omega, \varepsilon})=
\cO(\gamma\cdot\beta, Y_{\gamma\cdot\beta, K, \gamma \Omega, \varepsilon}, Z_{\id, \gamma \Omega, \varepsilon})$.
Thus $\gamma\cdot\beta$ also has the property according to which $\beta$ was chosen.

Now let $Y$ and $Z$ be neighbourhoods of $\alpha$ and $\id_R$ in $\Act(G, R)$ and $\Aut(R)$, respectively.
Shrinking $Y$ and $Z$ if necessary, we may assume that $Y=Y_{\alpha, K, \Omega, \varepsilon}$
and $Z=Z_{\id, \Omega, \varepsilon}$ for some $K$, $\Omega$, and $\varepsilon$ as above.
We claim that $Y\subseteq \overline{\cO(\alpha, Y, Z)}$. Let
$\theta\in Y$. Then $\theta \in Y_{\alpha, K, \Omega, \varepsilon'}$ for
some $0<\varepsilon'<\varepsilon$. Let $L$ be a finite subset of $G$ containing $K$, $\Upsilon$
a finite subset of $R$ containing $\Omega$, and $\delta\in (0, (\varepsilon-\varepsilon')/2)$.
By the theorem of Ocneanu in Section~1.4 of \cite{Ocneanu}, we can find unitaries $u_s\in R$ for $s\in L$
and $\gamma \in \Aut(R)$ such that $\| u_s-1\|_2<\delta/4$ and $(\gamma\cdot\beta)_s=\Ad u_s \circ \alpha_s$
for all $s\in L$. Note that
\begin{align*}
\| (\rho\cdot (\gamma\cdot \beta))_s(a)-(\rho\cdot \alpha)_s(a)\|_2 &=
\| \rho (u_s ) ((\rho\cdot\alpha)_s(a)) \rho(u_s)^*-(\rho\cdot \alpha)_s(a)\|_2 \\
&\leq 2\| \rho(u_s)-1\|_2
=2\| u_s-1\|_2 \\
&< \delta/2
\end{align*}
for all $s\in L$, $a\in \Upsilon$ and $\rho\in \Aut(R)$.
Set $Y_{\delta}=Y_{\gamma\cdot\beta, K, \Omega, \varepsilon'+\delta/2}$.
Then $\theta \in Y_{\delta}\subseteq \overline{\cO(\gamma\cdot\beta, Y_{\delta}, Z)}$.
Thus there exists an $\eta \in  \cO(\gamma\cdot\beta, Y_{\delta}, Z)\cap Y_{\theta, L, \Upsilon, \delta/2}$.
The above inequality shows that we can find a $\zeta\in \cO(\alpha, Y, Z)\cap Y_{\eta, L, \Upsilon, \delta/2}$.
Then $\zeta \in \cO(\alpha, Y, Z)\cap Y_{\theta, L, \Upsilon, \delta}$.
This proves the claim. Therefore $\alpha$ is a turbulent point.
\end{proof}

\begin{theorem}\label{T-amenable gen turb}
Suppose that $G$ is countably infinite and amenable. Then the action of $\Aut (R)$ on $\Fr (G,R)$ is turbulent.
\end{theorem}

\begin{proof}
Combine Lemmas~\ref{L-free dense}, \ref{L-free dense orbit}, \ref{L-meager}, and \ref{L-free turbulent}.
\end{proof}

As a consequence of Theorem~3.21 and Corollary~3.19 in \cite{COER} and Theorem~\ref{T-amenable gen turb},
we obtain the following.

\begin{theorem}
Suppose that $G$ is countably infinite and amenable.
Then no $\Aut (R)$-invariant dense $G_\delta$ subset of $\Fr (G,R)$ admits classification by
countable structures.
\end{theorem}

Note that when $G$ does not have property T the set of weakly mixing actions in $\Act (G,R)$
is a dense $G_\delta$ \cite{KP}, and so when $G$ is countably infinite and amenable
the free weakly mixing actions form an $\Aut (R)$-invariant dense $G_\delta$ subset of $\Fr (G,R)$
seeing that the latter is a dense $G_\delta$ in $\Act (G,R)$ by Lemma~\ref{L-free dense}.


\section{Flows and generic turbulence}

Let $(X, \mu )$ be a standard atomless probability space. The Polish space $\Act (G, X,\mu )$
of continuous $\mu$-preserving actions of a second countable locally compact group $G$ on $(X, \mu )$
with the weak topology can be identified with $\Act (G ,L^\infty (X,\mu ), \mu )$ as defined in Section~\ref{S-nonclassifiability}. The Polish group $\Aut (X,\mu )$ of $\mu$-preserving
transformations of $(X, \mu )$ can be viewed as $\Act (\Zb , X,\mu )$.
We write $\Erg (G, X,\mu )$ for the $G_\delta$ subset of ergodic actions in $\Act (G, X,\mu )$.
We say that an action $\alpha\in\Act (G , X,\mu )$ is {\it totally ergodic}
(resp.\ {\it totally weakly mixing}) if the single automorphism $\alpha_s$ is ergodic
(resp.\ weakly mixing) for every $s\in G\setminus\{ e \}$.

By \cite{RG} the set of weakly mixing actions in $\Act (\Rb, X,\mu )$
is a dense $G_\delta$. We have moreover the following.

\begin{lemma}\label{L-TWM}
The set of totally weakly mixing actions in $\Act (\Rb , X,\mu )$ is comeager.
\end{lemma}

\begin{proof}
Take an increasing sequence $\cP_1 \leq \cP_2 \leq\dots$ of finite measurable partitions
of $X$ whose union generates a dense subalgebra of the measure algebra. For each $n\in\Nb$ write
$W_n$ for the set of all actions $\alpha\in\Act (\Rb , X,\mu )$ such that there exists a
real number $r\geq n$ for which
$| \mu (\alpha_t (A)\cap B) - \mu(A)\mu(B) | < \varepsilon$
for all $A,B\in\cP$ and $t\in [r,r+n]$.
If we take $\beta$ in the proof of Theorem~4.2 in \cite{KP} to be a mixing
action (such as the Gaussian action associated to the left regular representation of $\Rb$),
then the argument there shows that we can approximate any action in $\Act (\Rb , X,\mu )$
with one in $W_n$. Thus $\bigcap_{n=1}^\infty W_n$ is a dense $G_\delta$, and it
consists entirely of totally weakly mixing actions, as is easy to see using the
characterization of weak mixing given by Corollary~1.6 of \cite{BerRos}.
\end{proof}

\begin{theorem}
The action of $\Aut (X,\mu )$ on $\Erg (\Rb , X,\mu )$ by conjugation is generically turbulent.
\end{theorem}

\begin{proof}
By \cite{RG} $\Erg (\Rb , X,\mu )$ is a dense $G_\delta$ in $\Act (\Rb , X,\mu )$,
and so by \cite[Thm.\ 3.21]{COER} it suffices to prove that the action of $\Aut (X,\mu )$
on $\Act (\Rb , X,\mu )$ is generically turbulent. Notice that the proof of
Lemma~\ref{L-turbulent point} works mutatis mutandis with $R$ replaced by $L^\infty (X,\mu )$,
the main difference being that the approximation of $\gamma_E$ by an inner automorphism
should be replaced by taking a tensor product of $\gamma_E$ with the identity automorphism.
Consequently there exists a turbulent point in $\Act (\Rb , X,\mu )$ with dense orbit. It thus
remains to verify that every orbit in $\Act (\Rb , X,\mu )$ is meager. So let $\alpha\in\Act (\Rb , X,\mu )$
and let us show that the orbit of $\alpha$ is meager. By Lemma~\ref{L-TWM} we may assume that
$\alpha$ is totally ergodic. Note that periodic flows are disjoint from totally ergodic
flows, as is easy to see by viewing joinings as unital completely positive maps according
to the discussion after Definition~\ref{D-disjoint}. Since the periodic flows are dense
in $\Act (\Rb , X,\mu )$ by the Rokhlin lemma for flows \cite[Lemma 11.1]{ETRDS},
we deduce in view of Lemma~\ref{L-disjoint} that the set
of actions in $\Act (\Rb , X,\mu )$ which are disjoint from $\alpha$ is a dense $G_\delta$,
so that the orbit of $\alpha$ is meager, as desired.
\end{proof}

In the above argument one can also use as a substitute for Lemma~\ref{L-TWM} the fact
that the set of flows which are ergodic at nonzero integral times is a dense $G_\delta$, which
can be deduced from the Rokhlin lemma for flows.

\end{document}